\documentclass[10pt,a4paper]{article}
\usepackage[english]{babel}
\usepackage{amssymb}
\usepackage{amsmath}
\usepackage{indentfirst}
\usepackage[dvips]{graphicx}
\usepackage{epsfig}
\usepackage{xcolor}

\usepackage[colorlinks,citecolor=blue,linkcolor=red]{hyperref}

\usepackage{color}

\usepackage{amsthm}
\usepackage{comment}
\usepackage{multirow}  
\usepackage{amsfonts}
\usepackage{mathtools} 
\usepackage{setspace} 

\newtheorem{cor}{Corollary}[section]

\newtheorem{lemma}[cor]{Lemma}
\newtheorem{teo}[cor]{Theorem}

\newcommand{\R}{\mathbb{R}}
\newcommand{\N}{\mathbb{N}}

\numberwithin{equation}{section}

\begin{document}

\title{On an elastic flow for parametrized curves in $\R^{n}$ suitable for numerical purposes}

\author{Paola Pozzi 
\thanks{Fakult\"at f\"ur Mathematik, Universit\"at Duisburg-Essen, 
Thea-Leymann-Stra\ss e 9, 
45127 Essen, Germany, 
\url{paola.pozzi@uni-due.de}} 
}

\date{\today}
\maketitle

\begin{abstract}
In \cite{PS22} a variant of the classical elastic flow  for closed curves in $\R^{n}$ was introduced, that is more suitable for numerical purposes. Here we investigate the long-time properties of such evolution demonstrating that the flow exists globally in time.
\end{abstract}

\\

\bigskip



\section{Introduction}

Let $f:I \to \R^{n}$, $f=f(x)$, $I=[0, 2\pi] \simeq S^{1}$, be the parametrisation of a closed (i.e. periodic) regular smooth curve, $ds=|f_{x}| dx$ its length element,
$\vec{\kappa}=f_{ss}$ its curvature vector, $\tau=f_{s}=\partial_{s} f=\frac{f_{x}}{|f_{x}|}$ its unit tangent vector.
Recall that the length $\mathcal{L}$, Dirichlet  $\mathcal{D}$ and bending energy $\mathcal{E}$ are defined by
\begin{align*}
\mathcal{L}(f) &:= \int_{I} ds= \int_{I} |f_{x}| dx, \qquad \mathcal{D}(f):= \frac{1}{2} \int_{I} |f_{x}|^{2} dx,\\
\mathcal{E}(f) & := \frac{1}{2} \int_{I} |\vec{\kappa}|^2 ds =\frac{1}{2} \int_{I} |\vec{\kappa}|^2 |f_{x}| dx.
\end{align*}
The study of evolution equations associated to the bending energy $\mathcal{E}$ has attracted a lot of attention in recent years: for motivation and extended references we refer here simply to \cite{DKS}, which has inspired a lot of the work presented here, and to a recent survey \cite{MPP21}, where several recent results are discussed.

For some positive given $\lambda >0$ we set
\begin{align*}
\mathcal{E}_{\lambda}(f):=\mathcal{E}(f)+ \lambda \mathcal{L}(f) \geq 0, \qquad \mathcal{D}_{\lambda}(f):=\mathcal{E}(f)+ \lambda \mathcal{D}(f) \geq0.
\end{align*}
The corresponding $L^{2}$-gradient flow are given by
\begin{align}\label{flowzero}
f_{t}  = -\nabla_{s}^{2} \vec{\kappa} -\frac{1}{2} |\vec{\kappa}|^{2} \vec{\kappa} +\lambda \vec{\kappa},
\end{align}
 for $\mathcal{E}_{\lambda}$ (cf. \cite{DKS}) and
 \begin{align}\label{1.2}
f_{t}& = -\nabla_{s}^{2} \vec{\kappa} -\frac{1}{2} |\vec{\kappa}|^{2} \vec{\kappa} +\lambda \frac{f_{xx}}{|f_{x}|}\\
&= -\nabla_{s}^{2} \vec{\kappa} -\frac{1}{2} |\vec{\kappa}|^{2} \vec{\kappa} +\lambda  \vec{\kappa} |f_{x}| + \lambda (|f_{x}|)_{s} \tau \notag
\end{align}
for $\mathcal{D}_{\lambda}$ (see Section~\ref{sec21} below). Here $\nabla_{s} \vec{\phi} =\partial_{s} \vec{\phi} -  \langle\partial_{s} \vec{\phi} , \tau\rangle \tau$ denotes the normal component of the derivative with respect to arc length for a given vector field $\vec{\phi}: [0,2\pi] \to \R^{n}$ along the curve.  Note that the flow associated to $\mathcal{E}_{\lambda}$ has a velocity vector that is entirely normal to the curve.

Motivated by  the numerical investigation undertaken in \cite{PS22}, we study  here the long-time existence properties for the $L^{2}$-gradient flow associated to $\mathcal{D}_{\lambda}$.

In \cite{PS22} we demonstrate that $\mathcal{E}_{\lambda}$  and  $\mathcal{D}_{\lambda}$ share common sets of critical points (in a suitable sense). However, from a numerical point of view, the minimization of the energy $\mathcal{D}_{\lambda}$ (via $L^{2}$-gradient flow) presents major advantages. Indeed the presence of a specific tangential component (see \eqref{1.2}) makes it possible to avoid grid-degeneration problems; moreover the numerical analysis is significantly simplified as opposed  to \cite{DD09} (see \cite{PS22} and related discussion in there. In \cite{PS22} one can also find interesting simulations of the evolution~\eqref{1.2}).

Long-time existence properties for the geometric $L^{2}$-gradient flow generated by $\mathcal{E}_{\lambda}$ are well known and investigated in \cite{DKS}. Since $\mathcal{D}$ dominates the length functional, in the sense that $ \mathcal{L}(f) \leq \sqrt{2\pi} \sqrt{2 \mathcal{D}(f)}$, one is inclined to believe that the $L^{2}$- gradient flow for $\mathcal{D}_{\lambda}$ should also exist for all times. However, this must be proved rigorously. This is the purpose of this work. 
Our method of proof is similar to that employed in \cite{DKS}, which is based on $L^{2}$- curvature estimates combined with Gagliardo-Nirenberg type inequalities.
However, our evolution is not geometric (the functional $\mathcal{D}_{\lambda}$ is not invariant under reparametrizations of the curves) and we must take care of the specific tangential component that appears in \eqref{1.2}. Therefore  new arguments are needed. In particular, upon observing  a strong relation between length element and tangential component, we exploit the second order PDE solved by the length element.
Our main result, whose proof is given in Section~\ref{sec3}, is the following:
\begin{teo}\label{mainteo} Let $\lambda \in (0, \infty)$ and  let $f_{0} :I \to \R^{n}$ be smooth and regular.
Assume that for any smooth regular initial data $f_{0}$ the flow \eqref{1.2} exists for some (small) time  $[0,T]$ and is smooth and regular on $[0,T] \times I$. Then:
the  
flow \eqref{1.2}  
has a global solution.
\end{teo}
This results hinges on a short-time existence result for the flow, which is outside the scope of the is paper and will be tackled elsewhere. On this matter let us here only remark, that a short-time existence results holds for \eqref{flowzero} (see \cite[\S~3]{DKS}, and also \cite{DLP-STE} where classical techniques are discussed in detail). The differences between \eqref{flowzero} and \eqref{1.2}  are to be found only in the lower order term multiplying $\lambda$, therefore it is safe to assume 
that a short-time existence result holds also in our setting. 
However, note that since the flow \eqref{1.2} is no longer geometric,  one can not factor out the degeneracy of the high order operator $\nabla_{s}^{2}$ in the usual way (i.e. by reparametrization), therefore some extra care must be taken in the arguments.

\bigskip
\noindent \textbf{Acknowledgements:} This project has been funded by the Deutsche Forschungsgemeinschaft (DFG, German Research Foundation)- Projektnummer: 404870139. 


\section{Preliminaries}\label{sec1}

We can write the (non-geometric) flow \eqref{1.2} as
$$ \partial_{t}f =\vec{V} + \varphi \tau$$
with normal component of the velocity vector given by
\begin{align}\label{V}
 \vec{V}:=-\nabla_{s}^{2} \vec{\kappa} -\frac{1}{2} |\vec{\kappa}|^{2} \vec{\kappa} +\lambda  \vec{\kappa} |f_{x}| =-\nabla_{s}^{2} \vec{\kappa} -\frac{1}{2} |\vec{\kappa}|^{2} \vec{\kappa} +\lambda \vec{w} 
 \end{align}
  where $$ \vec{w}:= \vec{\kappa} |f_{x}|,$$
and (scalar) tangential component
\begin{align}\label{tc}
 \varphi:=\lambda \langle \frac{f_{xx}}{|f_{x}|} , \tau \rangle = \lambda  \langle \frac{f_{xx}}{|f_{x}|^{2}} , f_{x}\rangle
= \lambda \frac{1}{|f_{x}|} \left( \frac{|f_{x}|^{2}}{2} \right)_{s} =\lambda (|f_{x}|)_{s}.
\end{align}
Here and in the following $\langle , \rangle$ denotes the Euclidean scalar product in $\R^{n}$.

\subsection{Evolution of geometric quantities}
\label{sec21}

For $\vec{\phi}$ any smooth normal field along $f$ and $h$ a scalar map we have that for any $m \in \N$ 
\begin{align}\label{cs1}
&\nabla_{s}(h \vec{\phi}) = (\partial_{s } h ) \vec{\phi} + h\nabla_{s} \vec{\phi}, \qquad \qquad \nabla_{s}^{m}(h \vec{\phi}) = \sum_{r=0}^{m}\binom{m}{r} \partial_{s}^{m-r}h \nabla_{s}^{r} \vec{\phi}
\end{align}
where recall that $\nabla_s \vec{\phi} = \partial_s \vec{\phi} - \langle \partial_s \vec{\phi}, \partial_s f \rangle \partial_s f$.
Similarly we write 
$\nabla_t \vec{\phi} = \partial_t \vec{\phi} - \langle \partial_t \vec{\phi}, \partial_s f \rangle \partial_s f$.

\begin{lemma}[\textbf{Evolution of geometric quantities}]\label{lemform} 
Let $f:[0, T)\times I \rightarrow \mathbb{R}^{n}$ be a smooth solution of $\partial_{t} f = \vec{V} + \varphi \tau$ for $t \in (0, T)$ with $\vec{V}$ the normal velocity. Given $\vec{\phi}$ any smooth normal field along $f$, the following formulas hold.
\allowdisplaybreaks{\begin{align}
\label{a}
\partial_{t}(ds)&=(\partial_{s} \varphi - \langle \vec{\kappa}, \vec{V} \rangle) ds \\
\label{b}
\partial_{t} \partial_{s}- \partial_{s}\partial_{t} &= (\langle \vec{\kappa}, \vec{V} \rangle -\partial_{s} \varphi) \partial_{s}\\
\label{c}
\partial_{t} \tau &= \nabla_{s} \vec{V} + \varphi \vec{\kappa}\\
\label{d}
  \partial_{t} \vec{\phi}&= \nabla_{t} \vec{\phi}- \langle \nabla_s \vec{V} + \varphi \vec{\kappa}, \vec{\phi} \rangle \tau \\ 
\label{e0}
\partial_{t} \vec{\kappa} & = \partial_{s} \nabla_{s} \vec{V} + \langle \vec{\kappa}, \vec{V}\rangle  \vec{\kappa} + \varphi \partial_{s} \vec{\kappa} \\
\label{e}
\nabla_{t} \vec{\kappa}& = \nabla_{s}^{2} \vec{V} + \langle \vec{\kappa}, \vec{V}\rangle  \vec{\kappa} + \varphi \nabla_{s} \vec{\kappa}\\
\label{f}
(\nabla_{t}\nabla_{s}- \nabla_{s}\nabla_{t}) \vec{\phi} &=(\langle \vec{\kappa},\vec{V} \rangle -\partial_{s }\varphi) \nabla_{s}\vec{\phi} +[\langle \vec{\kappa}, \vec{\phi}\rangle \nabla_{s}\vec{V} -\langle \nabla_{s} \vec{V}, \vec{\phi} \rangle \vec{\kappa}]. 
\end{align}} 
\end{lemma}
\begin{proof}
The proof follows by straightforward computation: see for instance \cite[Lemma 2.1]{DKS} or \cite[Lemma 2.1]{DP}. 
\end{proof}

The above lemma  holds in fact for \emph{any}  smooth evolution equation that can be written in the form $\partial_{t} f = \vec{V} + \varphi \tau$  with $\vec{V}$ the normal velocity. It  allows  to  compute the first variation of $\mathcal{D}_{\lambda}$ and derive \eqref{1.2} (cf.  \cite[Lemma~A.1]{DP}). Another important application is the following verification  that the energy $\mathcal{D}_{\lambda}$ decreases along the flow.

\smallskip

\noindent\textbf{Decrease in energy along the flow.}
To retrieve some  fundamental bounds it is important to notice that the energy decreases along the flow \eqref{1.2}.
Precisely, using \eqref{e}, \eqref{a}, and integration by parts we obtain
\begin{align*}
\frac{d}{dt} \mathcal{D}_{\lambda} (f) &=
\int_{I} \langle \vec{\kappa}, \nabla_{s}^{2} \vec{V} + \langle \vec{\kappa}, \vec{V}\rangle  \vec{\kappa} + \varphi \nabla_{s} \vec{\kappa} \rangle ds + \int_{I} \frac{1}{2} |\vec{\kappa}|^{2}
(\partial_{s} \varphi - \langle \vec{\kappa}, \vec{V} \rangle) ds + \lambda \int_{I} \langle f_{x}, f_{xt}\rangle dx\\
&= \int_{I} \langle \vec{\kappa}, \nabla_{s}^{2} \vec{V} + \frac{1}{2}\langle \vec{\kappa}, \vec{V}\rangle \vec{\kappa} \rangle ds  - \int_{I} \langle \lambda f_{xx}, f_{t} \rangle dx\\
&= \int_{I} \langle \nabla_{s}^{2} \vec{\kappa} +\frac{1}{2} |\vec{\kappa}|^{2} \vec{\kappa} , \vec{V}
\rangle ds -  \int_{I} \langle( \varphi |f_{x}| \tau + \lambda \vec{\kappa} |f_{x}|^{2} ), (\vec{V} +\varphi \tau) \rangle dx\\
&= -\int_{I} |\vec{V}|^{2} ds - \int_{I} \varphi^{2} ds \leq 0.
\end{align*}
\noindent \textbf{Uniform bounds along the flow.}
As a consequence of the energy decrease we infer that  the following uniform bounds hold for as long the the flow exists:
\begin{align*}
 \mathcal{D}_{\lambda} (f(t)) &\leq \mathcal{D}_{\lambda} (f(0)),\\
 \|\vec{\kappa}\|_{L^{2}(I)}^{2}(t) &= \int_{I} |\vec{\kappa}|^{2} ds \leq 2 \mathcal{D}_{\lambda} (f(0)),\\
 \mathcal{D}(f(t)) &=\frac{1}{2} \int_{I} |f_{x}|^{2} dx \leq \frac{1}{\lambda}\mathcal{D}_{\lambda} (f(0)),\\
 \mathcal{L}(f(t)) &\leq 2\sqrt{\frac{\pi}{\lambda} \mathcal{D}_{\lambda} (f(0))}.
\end{align*}
Moreover, as observed in \cite[(2.18)]{DKS}, since the curve is closed, the Poincar\'e inequality for $\partial_{s} f=\tau$ implies
\begin{align*}
2\pi  \leq \sqrt{\mathcal{L} (f(t))} \| \vec{\kappa}\|_{L^{2}},
\end{align*}
so that in view of the uniform bound from above for the curvature we obtain a uniform bound for the length from below. Hence along the flow we have that
\begin{align*}
 0< C^{-1} < \mathcal{L}(f(t)) \leq C
\end{align*}
where $C=C(\mathcal{D}_{\lambda} (f(0)), \lambda)$.
Last but not least we have that for any time $t$ where the flow is well defined
\begin{align}\label{wichtig}
\int_{0}^{t} \| \varphi\|_{L^{2}}^{2} dt' + \int_{0}^{t} \| \vec{V}\|_{L^{2}}^{2} dt'=
\int_{0}^{t} \int_{I} \varphi^{2} ds dt' +\int_{0}^{t} \int_{I} |\vec{V}|^{2} ds dt' \leq \mathcal{D}_{\lambda} (f(0)).
\end{align}

\noindent \textbf{PDEs for the length element $|f_{x}|$ and tangential component $\varphi$.}
Next we derive and investigate the evolution equation satisfied by the length element $|f_{x}|$ and the tangential component $\varphi$. Note that in view of \eqref{tc} there is a strong relation between the two of them.

We start by considering the length element.
For as long as $|f_{x}| \geq C >0$, classical embedding theory yields that
\begin{align*}
\|\, |f_{x}|\, \|_{C^{0}[0, 2\pi]} \leq  \int_{I}  ||f_{x}|_{x} | dx + \frac{1}{2\pi} \int_{I} |f_{x}| dx.
\end{align*}
Using the uniform bounds on the  length and the fact that $\lambda$ is a fixed contant, it follows then that
\begin{align}\label{stella1}
\| \,|f_{x}|\, \|_{L^{\infty}} \leq C + \int_{I} \frac{|\varphi|}{\lambda} ds \leq C (1 +  \| \varphi \|_{L^{2}}) 
\end{align}
where $C=C(\mathcal{D}_{\lambda} (f(0)), \lambda)$.

Next, let us have a closer look at the evolution equation satisfied by the length element.
Using the definition of the tangential component \eqref{tc} we can write
\begin{align*}
\partial_{t} (|f_{x}|) &= \langle \tau, f_{tx} \rangle = \varphi_{x} - \langle \vec{\kappa}, f_{t} \rangle |f_{x}|=\lambda ((|f_{x}|)_{s})_{x} - \langle \vec{\kappa}, \vec{V}\rangle |f_{x}|\\
& =\frac{\lambda}{|f_{x}|} (|f_{x}|)_{xx} + \lambda (|f_{x}|)_{x} \left( \frac{1}{|f_{x}|}\right)_{x}
- \langle \vec{\kappa}, \vec{V}\rangle |f_{x}|.
\end{align*}
Note that to apply a maximum principle we would need some uniform bounds on $\vec{V}$ and $\vec{\kappa}$, which at the moment are out of reach.

Now we turn to the tangential component \eqref{tc}. We first explain some useful notation. 
In the following we write $B_{2}^{a,c} (\varphi)$ for any linear combination of terms of type
\begin{align*}
(\partial_{s}^{i_{1}}\varphi ) ( \partial_{s}^{i_{2}}\varphi), \qquad   \text{ with } i_{1}+i_{2} =a \text{ and }
\max i_{j} \leq c
\end{align*}
with universal, constant coefficients. Notice that $a$ records the total number of derivatives and $c$ gives the highest number of derivatives falling on one factor. We have that $ \partial_{s} B_{2}^{a,c} (\varphi)=   B_{2}^{a+1,c+1} (\varphi)$.
Similarly we write $ M_{2}^{a,c}(\langle\vec{\kappa}, \vec{V}\rangle, \varphi)$ for any linear combination of terms of type
\begin{align*}
\partial_{s}^{i_{1}}( \langle \vec{\kappa}, \vec{V} \rangle) \,  \partial_{s}^{i_{2}} \varphi, \qquad   \text{ with } i_{1}+i_{2} =a \text{ and }
\max i_{j} \leq c
\end{align*}
with universal, constant coefficients. Note that $ \partial_{s} M_{2}^{a,c} (\langle\vec{\kappa}, \vec{V}\rangle, \varphi)=   M_{2}^{a+1,c+1}(\langle\vec{\kappa}, \vec{V}\rangle, \varphi)$. 

Using \eqref{b}, the previous computations, and recalling that $\varphi=\lambda(|f_{x}|)_{s}$ we immediately infer 
\begin{align}\label{PDEvarphi}
\partial_{t} \varphi &= \lambda \partial_{t} \partial_{s}  (|f_{x}|) = \lambda \partial_{s} \partial_{t} (|f_{x}|) + (\langle \vec{\kappa}, \vec{V} \rangle - \varphi_{s}) \partial_{s} (\lambda|f_{x}|) \notag\\
&=\lambda \partial_{s} (  \varphi_{x} - \langle \vec{\kappa}, f_{t} \rangle |f_{x}|)
+(\langle \vec{\kappa}, \vec{V} \rangle - \varphi_{s})  \varphi \notag\\
&=\lambda (|f_{x}| \varphi_{s})_{s} - \lambda (\langle \vec{\kappa}, \vec{V} \rangle |f_{x}|)_{s} +(\langle \vec{\kappa}, \vec{V} \rangle - \varphi_{s})  \varphi. 
\end{align}
This gives also
\begin{align*}
\partial_{t} \varphi &= \varphi \varphi_{s} + \lambda |f_{x}| \varphi_{ss} - (\langle \vec{\kappa}, \vec{V} \rangle)_{s} \lambda |f_{x}| - \langle \vec{\kappa}, \vec{V} \rangle \varphi +  \langle \vec{\kappa}, \vec{V} \rangle \varphi -\varphi \varphi_{s} \\
&= \lambda |f_{x}| \varphi_{ss} - (\langle \vec{\kappa}, \vec{V} \rangle)_{s} \lambda |f_{x}|.
\end{align*}
Next we compute in a similar manner
\begin{align}\label{PDEvarphis}
\partial_{t} ( \varphi_{s} ) &= \partial_{t} \partial_{s}\varphi = \partial_{s} \partial_{t} \varphi
+( \langle \vec{\kappa}, \vec{V} \rangle - \varphi_{s}) \varphi_{s}  \notag\\
&= (\lambda |f_{x}| \varphi_{ss})_{s} - \left( (\langle \vec{\kappa}, \vec{V} \rangle)_{s} \lambda |f_{x}|  \right)_{s} +( \langle \vec{\kappa}, \vec{V} \rangle - \varphi_{s}) \varphi_{s}\\
&=(\lambda |f_{x}| \varphi_{ss})_{s} - \left( (\langle \vec{\kappa}, \vec{V} \rangle)_{s} \lambda |f_{x}|  \right)_{s} + M_{2}^{1,1} (\langle\vec{\kappa}, \vec{V}\rangle, \varphi) + B_{2}^{2,1} (\varphi) \notag.
\end{align}
This gives also
\begin{align*}
\partial_{t} ( \varphi_{s} ) &= \lambda |f_{x}|  \partial_{s}^{3} \varphi + \varphi \partial_{s}^{2}\varphi  -\lambda |f_{x}| \partial_{s}^{2} (\langle \vec{\kappa}, \vec{V} \rangle)  
- \varphi \partial_{s} (\langle \vec{\kappa}, \vec{V} \rangle)   + \langle \vec{\kappa}, \vec{V} \rangle \varphi_{s}
- (\varphi_{s})^{2}\\
& =\lambda |f_{x}|  \partial_{s}^{3} \varphi -\lambda |f_{x}| \partial_{s}^{2} (\langle \vec{\kappa}, \vec{V} \rangle)  + M_{2}^{1,1} (\langle\vec{\kappa}, \vec{V}\rangle, \varphi) + B_{2}^{2,2} (\varphi).
\end{align*} 
Proceeding inductively  one finds:
\begin{lemma}
 For any $m \in \mathbb{N}$ we have
\begin{align}\label{PDEvarphims}
\partial_{t} (& \partial_{s}^{m}\varphi )  \notag \\
&=(\lambda |f_{x}|  \partial_{s}^{m+1} \varphi )_{s} 
- \left(  \lambda |f_{x}| \partial_{s}^{m} (\langle \vec{\kappa}, \vec{V} \rangle) \right)_{s} +M_{2}^{m,m} (\langle\vec{\kappa}, \vec{V}\rangle, \varphi) +B_{2}^{m+1,m} (\varphi)\\
&= \lambda |f_{x}|  \partial_{s}^{m+2} \varphi
-  \lambda |f_{x}| \partial_{s}^{m+1} (\langle \vec{\kappa}, \vec{V} \rangle)
 +M_{2}^{m,m} (\langle\vec{\kappa}, \vec{V}\rangle, \varphi) +B_{2}^{m+1,m+1} (\varphi)  . \notag
\end{align}
\end{lemma}

\noindent \textbf{On the curvature vector and its derivatives.}
For geometric terms such as the curvature vector and its derivatives we will make use of the following lemma, which is a straight-forward generalisation of \cite[Lemma 2.2]{DKS}.
\begin{lemma}\label{lempartint}
Suppose $\partial_{t}f =\vec{V}+\varphi \tau$ on $(0,T) \times I$. Let $\vec{\phi}$ be a normal vector field along $f$ and $Y=\nabla_{t} \vec{\phi} + \nabla_{s}^{4} \vec{\phi}$.
Then
\begin{align}\label{eqgen0}
\frac{d}{dt} \frac{1}{2}\int_{I} |\vec{\phi}|^{2} ds + \int_{I}
|\nabla_{s}^{2} \vec{\phi}|^{2 } ds & = 
\int_{I} \langle Y + \frac{1}{2} \vec{\phi} \,  \varphi_{s}, \vec{\phi} \rangle ds - \frac{1}{2} \int_{I} |\vec{\phi}|^{2} \langle \vec{\kappa}, \vec{V} \rangle ds ,
\end{align}
\end{lemma}
\begin{proof}
 The claim follows  using \eqref{a} and integration by parts.
\end{proof}

As in \cite[Sec.~3]{DLP-LTE}, \cite[Lem.2.3]{DKS} and \cite[Sec.~3]{DP} we denote by  $ \vec{\phi}_1 *  \vec{\phi}_2 * \cdots * \vec{\phi}_k$ the product of $k$ normal vector fields $\vec{\phi}_i$ ($i=1,..,k$) defined as 
$\langle \vec{\phi}_1,\vec{\phi}_2 \rangle \cdot .. \cdot\langle \vec{\phi}_{k-2},\vec{\phi}_{k-1} \rangle \vec{\phi}_{k}$
if $k$ is odd and as 
$\langle \vec{\phi}_1,\vec{\phi}_2 \rangle \cdot .. \cdot \langle \vec{\phi}_{k-1},\vec{\phi}_{k} \rangle$, if $k$ is even. 
The expression $P_b^{a,c}(\vec{\kappa})$ stands for any linear combination of terms of the type
\[
 (\nabla_{s}^{i_1} \vec{\kappa})   * \cdots * (\nabla_{s}^{i_b} \vec{\kappa})  \text{ with }i_1 + \ldots + i_b = a \text{ and }\max i_j \leq c
\]
with universal, constant coefficients. Thus $a$ gives the total number of derivatives, $b$ denotes the number of factors and $c$ gives a bound on the highest number of derivatives falling on one factor. Using \eqref{cs1} we observe that for $b \in \mathbb{N}$, $b $  odd, we have
$ \nabla_s P^{a,c}_b (\vec{\kappa})= P^{a+1,c+1}_b (\vec{\kappa}).$
With a slight abuse of notation, 
$|P^{a,c}_b(\vec{\kappa})|$ 
denotes any linear combination with non-negative coefficients of terms of  type 
$$|\nabla_{s}^{i_1} \vec{\kappa}|\cdot |\nabla_{s}^{i_2} \vec{\kappa}| \cdot ... \cdot | \nabla_{s}^{i_{b}} \vec{\kappa}| \mbox{ with }i_1 + \dots +i_{b}= a \mbox{ and } \max i_{j} \leq c \, .$$
Similarly we write 
$Q_b^{a,c}(\vec{\kappa}, \vec{w})$  for any linear combination of terms of the type
\[
 (\nabla_{s}^{i_1} \vec{w}) *(\nabla_{s}^{i_2} \vec{\kappa}) * \cdots * (\nabla_{s}^{i_b} \vec{\kappa} ) \text{ with }i_1 + \ldots + i_b = a \text{ and }\max i_j \leq c
\]
with universal, constant coefficients. Also in this case for odd $b \in \mathbb{N}$ we have
$ \nabla_s Q^{a,c}_b (\vec{\kappa}, \vec{w})= Q^{a+1,c+1}_b (\vec{\kappa}, \vec{w}).$
For sums we write 
\begin{equation}
\sum_{\substack{[[a,b]] \leq [[A,B]]\\c\leq C}} P^{a,c}_{b} (\vec{\kappa}) 
:=\sum_{a=0}^{A}\sum_{b=1}^{2A+B-2a} \sum_{c=0}^C\text{ } P_{b}^{a,c}(\vec{\kappa})
\text{.}
\label{eq:sum_P^a_b}
\end{equation}
Similarly we set
$
\sum_{\substack{[[a,b]] \leq [[A,B]]\\c\leq C}}| P^{a,c}_{b} (\vec{\kappa})| : = \sum_{a=0}^{A} \sum_{b=1}^{2A+B-2a} \sum_{c=0}^{C} | P^{a,c}_{b}(\vec{\kappa}) |\, .
$

With this notation we can state the following result, which relates the operator $\nabla_{s}^{m}$ to the full derivative $\partial_{s}^{m}$. Loosely speaking one can say that  $\partial_{s}^{m} \vec{\kappa}$ and 
$\nabla_{s}^{m} \vec{\kappa}$ ``are the same'' up to lower order terms.
\begin{lemma}
\label{Qlemma}
We have the identities
\begin{align*}
& \partial_{s} \vec{\kappa}= \nabla_{s}\vec{\kappa} -|\vec{\kappa}|^{2}\tau,\\
& \partial_{s}^{m} \vec{\kappa} = \nabla_{s}^{m} \vec{\kappa} +\tau \sum_{\substack{[[a,b]] \leq [[m-1,2]]\\c \leq m-1, \ b \   even}} P^{a,c}_{b} 
(\vec{\kappa}) + \sum_{\substack{[[a,b]] \leq [[m-2,3]]\\c \leq m-2\ b \  odd}} P^{a,c}_{b} (\vec{\kappa})  \quad \text{ for } m\geq 2 \, .
\end{align*}
\end{lemma}
\begin{proof} The proof can be found for instance in \cite[Lemma~ 4.5]{DP} (see also \cite[Lemma 2.6]{DKS}).
The first claim is obtained directly using that 
$$\partial_s \vec{\kappa}= \nabla_s \vec{\kappa} + \langle \partial_s \vec{\kappa}, \tau \rangle \tau  = \nabla_s \vec{\kappa} - |\vec{\kappa}|^2 \tau \, .$$ 
The second claim follows by induction.
\end{proof}

We are now able to describe in detail the evolution of the curvature vector and its derivatives. 
\begin{lemma}\label{evolcurvature}
Suppose $\partial_t f= -\nabla_s^2 \vec{\kappa} - \frac12 |\vec{\kappa}|^2 \vec{\kappa} +  \lambda \vec{w}+ \varphi \tau$, where $\lambda =\lambda(t)$. Then for $m  \in \mathbb{N}_0$ we have
\begin{align*} \nabla_t \nabla_s^m \vec{\kappa}+ \nabla_s^4 \nabla_s^m \vec{\kappa}&= P^{m+2,m+2}_3 (\vec{\kappa})+ \lambda (\nabla_s^{m+2} \vec{w} + Q^{m,m}_3 (\vec{\kappa}, \vec{w})) +P^{m,m}_5 (\vec{\kappa}) + \varphi \nabla_s^{m+1} \vec{\kappa}.
\end{align*}
\end{lemma}

\begin{proof}
For $m=0$ the claim follows directly from \eqref{e}. 
For $m=1$  it follows using \eqref{f} and  \eqref{e}, namely 
\allowdisplaybreaks{\begin{align*}
\nabla_t \nabla_s \vec{\kappa}+ \nabla_s^5 \vec{\kappa} & = \nabla_s \nabla_t  \vec{\kappa} + (\langle \vec{\kappa},\vec{V} \rangle -\partial_{s }\varphi) \nabla_{s}\vec{\kappa} +[\langle \vec{\kappa}, \vec{\kappa}\rangle \nabla_{s}\vec{V} -\langle \nabla_{s} \vec{V}, \vec{\kappa} \rangle \vec{\kappa}] + \nabla_s^5 \vec{\kappa}\\
&= \nabla_{s} (-\nabla_s^4 \vec{\kappa}  + P^{2,2}_3 (\vec{\kappa})+ \lambda (\nabla_s^{2} \vec{w} + Q^{0,0}_3 (\vec{\kappa}, \vec{w})) +P^{0,0}_5 (\vec{\kappa}) + \varphi \nabla_s \vec{\kappa})\\
& \quad 
+ (\langle \vec{\kappa},\vec{V} \rangle -\partial_{s }\varphi) \nabla_{s}\vec{\kappa} +[\langle \vec{\kappa}, \vec{\kappa}\rangle \nabla_{s}\vec{V} -\langle \nabla_{s} \vec{V}, \vec{\kappa} \rangle \vec{\kappa}] + \nabla_s^5 \vec{\kappa}.
\end{align*}
Since  $\nabla_{s} (\varphi \nabla_s \vec{\kappa}) = \varphi_{s} \nabla_s \vec{\kappa} + \varphi \nabla_s^{2} \vec{\kappa}$ and 
\allowdisplaybreaks{\begin{align*}
\langle \vec{\kappa}, \vec{V}\rangle  \nabla_{s} \vec{\kappa} & =  P_3^{3,2}(\vec{\kappa})+  P_5^{1,1}(\vec{\kappa})+  \lambda Q_3^{1,1}(\vec{\kappa}, \vec{w}) \, \\
\mbox{and }\qquad |\vec{\kappa}|^2 \nabla_s \vec{V} & = P_3^{3,3}(\vec{\kappa})+  P_5^{1,1}(\vec{\kappa})+  \lambda Q_3^{1,1}(\vec{\kappa}, \vec{w}),
\end{align*}}
we infer
\begin{align*}
\nabla_t \nabla_s \vec{\kappa}+ \nabla_s^5 \vec{\kappa} &= P^{3,3}_3 (\vec{\kappa})+ \lambda (\nabla_s^{3} \vec{w} + Q^{1,1}_3 (\vec{\kappa}, \vec{w})) +P^{1,1}_5 (\vec{\kappa}) + \varphi \nabla_s^{2} \vec{\kappa} \,,
\end{align*}
noticing that  the terms appearing in  $P^{3,2}_3 (\vec{\kappa})$
can be collected in $P^{3,3}_3 (\vec{\kappa})$. }
The general statement follows with an induction argument.
\end{proof}

\subsection{Interpolation inequalities and embeddings}

We start by recalling some fundamental interpolation inequalities.
Consider the scale invariant norms for $k \in \mathbb{N}_{0}$ and $p \in [1, \infty)$
\begin{equation*}
\| \vec{\kappa}\|_{k,p} := \sum_{i=0}^{k} \| \nabla_s^{i} \vec{\kappa}\|_{p} \quad \mbox{ with } \quad \| \nabla_s^{i} \vec{\kappa} \|_{p} := \mathcal{L}(f)^{i+1-1/p} \Big( \int_I |\nabla_s^{i} \vec{\kappa}|^p \, ds \Big)^{1/p} \, ,
\end{equation*}
(cf. \cite{DKS}) and the usual $L^p$- norm $\| \nabla_s^{i} \vec{\kappa} \|_{L^p}^p :=  \int_I |\nabla_s^{i} \vec{\kappa}|^p \, ds $.  

Most of the following results can be found in several papers (e.g. \cite{DKS,Lin,DP}). We provide reference to the papers where  complete proofs can be found.

\begin{lemma}[Lemma~4.1 \cite{DP}]\label{leminter}
Let $f: I \rightarrow \mathbb{R}^n$ be a smooth regular curve. Then for all $k \in \mathbb{N}$, $p \geq 2$ and $0 \leq i<k$ we have
\begin{equation*}
\| \nabla_s^{i} \vec{\kappa}\|_{p} \leq C \|\vec{\kappa}\|_{2}^{1-\alpha} \|\vec{\kappa}\|_{k,2}^{\alpha} \, ,  
\end{equation*}
with $\alpha= (i+\frac12 -\frac{1}{p})/k$ and $C=C(n,k,p)$.
\end{lemma}

\begin{cor}[Corollary~4.2 \cite{DP}]\label{corinter}
Let $f: I \rightarrow \mathbb{R}^n$ be a smooth regular curve. Then for all $k \in \mathbb{N}$ we have
\begin{equation*}
\|\vec{\kappa}\|_{k,2} \leq C ( \|\nabla^{k}_{s}\vec{\kappa}\|_{2} + \|\vec{\kappa}\|_{2}) \, ,
\end{equation*}
with $C=C(n,k)$.
\end{cor}

\begin{lemma}[Lemma~3.4 \cite{DLP2}]\label{lemineqsh} 
Let $f: I \rightarrow \mathbb{R}^n$ be a smooth regular curve. For any $a,c,\ell\in \mathbb{N}_0$, $b \in \mathbb{N}$, $b\geq 2$, $c \leq \ell+2$ and $a < 2(\ell+2)$ we find
\begin{equation}\label{interineq1sh}
\int_I |P^{a,c}_{b} (\vec{\kappa})| \, ds \leq C \mathcal{L}(f)^{1-a-b} \|\vec{\kappa}\|_{2}^{b-\gamma} \| \vec{\kappa}\|_{\ell+2,2}^{\gamma} \, ,
\end{equation}
with $\gamma=(a+\frac12 b-1)/(\ell+2)$ and $C=C(n,\ell,a, b)$. Further if  $a+\frac12 b<2\ell+5$, then for any $\varepsilon>0$
\begin{align}\label{interineq1sh2}
 \int_I |P^{a,c}_{b}(\vec{\kappa})| \, ds & \leq \varepsilon \int_I |\nabla_s^{\ell+2} \vec{\kappa}|^2 \, ds +C \varepsilon^{-\frac{\gamma}{2-\gamma}}  (\|\vec{\kappa}\|^2_{L^{2}})^{\frac{b-\gamma}{2-\gamma}}
  + C  \mathcal{L}(f)^{1-a-\frac{b}{2}}  \|\vec{\kappa}\|_{L^{2}}^{b}  \, ,
\end{align}
with $C=C(n,\ell,a, b)$.
\end{lemma}

 We finish this section with some important results that are based on classical embedding theory.
\begin{lemma}[Lemma 2.7 \cite{DKS}]\label{lememb}
Assume that the bounds $\|\vec{\kappa}\|_{L^{2}} \leq \Lambda_{0}$ and $\|\nabla_{s}^{m} \vec{\kappa}\|_{L^{1}} \leq \Lambda_{m}$ for $m \geq 1$. Then for any $m \geq 1$ one has
\begin{align}
\| \partial_{s}^{m-1} \vec{\kappa} \|_{L^{\infty}} + \| \partial_{s}^{m} \vec{\kappa}\|_{L^{1}} \leq c_{m}(\Lambda_{0}, \ldots, \Lambda_{m}).
\end{align}
\end{lemma}

\begin{lemma}\label{hilfsatz}
For  any smooth scalar map $h:I \to \R$ and  normal vector field  $\vec{\phi}:I \to \R^{n}$ along $f$ we have that
\begin{align*}
\| \vec{\phi} \|_{\infty} \leq C (\| \vec{\phi} \|_{L^{2}} + \|\nabla_{s} \vec{\phi} \|_{L^{2}})\\
\| h \|_{\infty} \leq C (\| h \|_{L^{2}} + \|\partial_{s} h \|_{L^{2}})\
\end{align*}
where $C=C(\frac{1}{\mathcal{L}(f)})$. 
\end{lemma}
\begin{proof}
The proof of both statements can be found in the proof of \cite[Lemma~3.7]{DLP-LTE}. It is an  application of classical embedding theory to the map $|\vec{\phi}|^{2}$ respectively $h^{2}$.
\end{proof}
More generally we  can state the following.
\begin{lemma}[Lemma 3.7 \cite{DLP-LTE}]\label{lem:trickbdry}
We have that for any $x \in [0,1]$ there holds
\begin{align}\label{trickboundary}
| P^{a,c}_{b}(\vec{\kappa})(x)|^{2} \leq C  \int_{I} \big( |P^{2a+1,c+1}_{2b}(\vec{\kappa})| + |P^{2a,c}_{2b} (\vec{\kappa})|\big) ds, \quad \text{ if $b$ is odd},\\ \label{trickboundary2}
| P^{a,c}_{b}(\vec{\kappa})(x)| \leq C  \int_{I} \big(|P^{a+1,c+1}_{b}(\vec{\kappa})| + |P^{a,c}_{b} (\vec{\kappa})| \big) ds, \quad \text{ if $b$ is even},
\end{align}
where $C=C(\frac{1}{\mathcal{L}(f)})$.
\end{lemma}

\section{Long-time existence}\label{sec3}
This section is devoted to the proof of Theorem~\ref{mainteo}.
By assumption we know 
that given any smooth regular initial data $f_{0}$, 
there exists a smooth regular solution $f: [0,T) \times I \to \R^{n}$ of \eqref{1.2} with $f(0, \cdot) = f_{0}$.
Assume by contradiction that  the solution does not exist globally in time and  let $0<T<\infty$ be the maximal time. Recall that on $[0,T)$ the uniform bounds   listed in Section~\ref{sec21} hold (with constants that depend on $\lambda$ and the initial energy but not on $T$). In particular 
\begin{align}\label{star}
\| \vec{\kappa} (t)\|_{L^{2}} \leq C, \qquad \frac{1}{C} \leq \mathcal{L}(f(t)) \leq C, \qquad t \in [0, T)
\end{align}
with $C =C(\lambda, \mathcal{D}_{\lambda}(f_{0}))$. In the following a constant $C$ may vary from line to line, but we will indicate  what it depends on.

Our first task is to derive uniform bounds for $\vec{\kappa}$, $\varphi$ and their derivatives. This is performed in several steps, using an induction procedure.

\bigskip
\noindent\textbf{First Step - Part A: bound on $\| \nabla_{s} \vec{\kappa}\|_{L^{2}}$.}
Recalling \eqref{V}, using Lemma~\ref{lempartint} with $\vec{\phi}:=\nabla_{s}\vec{\kappa}$, Lemma~\ref{evolcurvature}, and exploiting the fact that $ \varphi \langle \nabla_{s}^{2} \vec{\kappa}, \nabla_{s} \vec{\kappa} \rangle +  \frac{1}{2} \varphi_{s} |\nabla_{s} \vec{\kappa}|^{2} = \partial_{s} \left(  \frac{1}{2} \varphi |\nabla_{s} \vec{\kappa}|^{2}\right)$
we obtain
\allowdisplaybreaks{
\begin{align*}
\frac{d}{dt} &\left( \frac{1}{2}\int_{I} |\nabla_{s} \vec{\kappa}|^{2} ds\right) +
\int_{I} |\nabla_{s}^{3} \vec{\kappa}|^{2} ds + \frac{1}{2}\int_{I} |\nabla_{s} \vec{\kappa}|^{2} ds \\
&=\int_{I}\langle (\nabla_{t}+ \nabla_{s}^{4}) \nabla_{s} \vec{\kappa},\nabla_{s} \vec{\kappa} \rangle 
+ \frac{1}{2} \varphi_{s} |\nabla_{s} \vec{\kappa}|^{2} ds - \frac{1}{2} \int_{I} |\nabla_{s} \vec{\kappa}|^{2} \langle  \vec{\kappa}, \vec{V} \rangle ds + \frac{1}{2}\int_{I} |\nabla_{s} \vec{\kappa}|^{2} ds\\
&= \int_{I} \langle  P^{3,3}_3 (\vec{\kappa})+ \lambda (\nabla_s^{3} \vec{w} + Q^{1,1}_3 (\vec{\kappa}, \vec{w})) +P^{1,1}_5 (\vec{\kappa}), \nabla_{s} \vec{\kappa} \rangle ds \\
& \qquad- \frac{1}{2} \int_{I} |\nabla_{s} \vec{\kappa}|^{2} \langle  \vec{\kappa}, \vec{V} \rangle ds
+ \frac{1}{2}\int_{I} |\nabla_{s} \vec{\kappa}|^{2} ds
\\
& =\int_{I} P_{4}^{4,3} (\vec{\kappa}) + P^{2,1}_6 (\vec{\kappa})  + P_{2}^{2,1} (\vec{\kappa})
+ P_{4}^{4,2} (\vec{\kappa}) ds +
 \lambda \int_{I} \langle \nabla_s^{3} \vec{w} + Q^{1,1}_3 (\vec{\kappa}, \vec{w}), \nabla_{s} \vec{\kappa} \rangle ds\\
 & =\int_{I} P_{4}^{4,3} (\vec{\kappa}) + P^{2,1}_6 (\vec{\kappa}) + P_{2}^{2,1} (\vec{\kappa})  ds +
 \lambda \int_{I} \langle \nabla_{s} \vec{w}, \nabla_{s}^{3} \vec{\kappa} \rangle ds + \lambda  \int_{I} \langle Q^{1,1}_3 (\vec{\kappa}, \vec{w}), \nabla_{s} \vec{\kappa} \rangle ds\\
& =J_{1}+J_{2}+J_{3},
\end{align*} }
where we have used integration by parts in the last step and absorbed the terms $ P_{4}^{4,2} (\vec{\kappa}) $ into $ P_{4}^{4,3} (\vec{\kappa}) $.
By applying Lemma~\ref{lemineqsh} and \eqref{star} we find
\begin{align*}
|J_{1}|&=|\int_{I} P_{4}^{4,3} (\vec{\kappa}) + P^{2,1}_6 (\vec{\kappa}) + P_{2}^{2,1} (\vec{\kappa}) ds | \leq \epsilon \int_{I} |\nabla_{s}^{3} \vec{\kappa}|^{2} ds + C(\epsilon, n, \lambda, \mathcal{D}_{\lambda}(f_{0}) ) .
\end{align*}
Since
\begin{align}\label{nablaw}
\nabla_{s}(\vec{w})= \nabla_{s} (|f_{x}| \vec{\kappa})= |f_{x}| \nabla_{s} \vec{\kappa} + (|f_{x}|)_{s} \vec{\kappa}= |f_{x}| \nabla_{s} \vec{\kappa} +\frac{\varphi}{\lambda} \vec{\kappa}
\end{align}
we can write
\begin{align*}
J_{2} =\lambda \int_{I} \langle \nabla_{s} \vec{w}, \nabla_{s}^{3} \vec{\kappa} \rangle ds
=\lambda \int_{I} |f_{x}| \langle  \nabla_{s} \vec{\kappa},  \nabla_{s}^{3} \vec{\kappa}\rangle ds +
 \int_{I} \varphi \langle   \vec{\kappa},  \nabla_{s}^{3} \vec{\kappa}\rangle ds=J_{2,1}+J_{2,2}.
\end{align*}
We have
\begin{align*}
J_{2,1} &\leq \epsilon \int_{I} |\nabla_{s}^{3} \vec{\kappa}|^{2} ds + C_{\epsilon} \| |f_{x}|\|_{\infty}^{2} \int_{I} | \nabla_{s} \vec{\kappa}|^{2} ds \\
&\leq 
\epsilon \int_{I} |\nabla_{s}^{3} \vec{\kappa}|^{2} ds   +
C_{\epsilon}(1 +\| \varphi\|_{L^{2}}^{2})\int_{I} | \nabla_{s} \vec{\kappa}|^{2} ds
\end{align*}
where we have used \eqref{stella1} in the last step. Note that here $C_{\epsilon}= C(\epsilon, \lambda, \mathcal{D}_{\lambda}(f_{0}))$. 
Next, we compute
\begin{align*}
J_{2,2} &=\int_{I} \varphi \langle   \vec{\kappa},  \nabla_{s}^{3} \vec{\kappa}\rangle ds \leq
\| \vec{\kappa}\|_{L^{\infty}} \|\varphi\|_{L^{2}} \| \nabla_{s}^{3} \vec{\kappa}\|_{L^{2}}
 \leq \epsilon\int_{I} |\nabla_{s}^{3} \vec{\kappa}|^{2} ds + C_{\epsilon} \| \vec{\kappa}\|^{2}_{L^{\infty}} \|\varphi\|_{L^{2}}^{2}. 
\end{align*}
Since
$
\| \vec{\kappa}\|_{L^{\infty}}^{2} \leq C (\| \vec{\kappa}\|_{L^{2}}^{2} + \| \nabla_{s} \vec{\kappa}\|_{L^{2}}^{2}) $ 
by Lemma~\ref{hilfsatz} and~\eqref{star} we obtain
\begin{align*}
J_{2,2} \leq \epsilon\int_{I} |\nabla_{s}^{3} \vec{\kappa}|^{2} ds + C_{\epsilon} (1 + \| \nabla_{s} \vec{\kappa}\|_{L^{2}}^{2}) \|\varphi\|_{L^{2}}^{2}.
\end{align*}
Using the definition of $Q^{1,1}_3 (\vec{\kappa}, \vec{w})$, $\vec{w}$, \eqref{nablaw}, and \eqref{interineq1sh} we observe that
\begin{align*}
J_{3}& =\lambda  \int_{I} \langle Q^{1,1}_3 (\vec{\kappa}, \vec{w}), \nabla_{s} \vec{\kappa} \rangle ds=\int_{I} \lambda|f_{x}| P_{4}^{2,1}(\vec{\kappa}) + \varphi P_{4}^{1,1}(\vec{\kappa}) ds\\
& \leq |\lambda|\| |f_{x}| \|_{L^{\infty}} \int_{I} | P_{4}^{2,1}(\vec{\kappa})| ds +\| \varphi\|_{L^{2}} \left(\int_{I} |P_{8}^{2,1} (\vec{\kappa})| ds \right)^{\frac{1}{2}}\\
& \leq C |\lambda| \| |f_{x}| \|_{L^{\infty}} \mathcal{L}(f)^{-5} \| \vec{\kappa}\|_{2}^{3} \| \vec{\kappa}\|_{3,2}
+ C\| \varphi\|_{L^{2}} \mathcal{L}(f)^{-9/2} \| \vec{\kappa}\|_{2}^{\frac{8-5/3}{2}} \| \vec{\kappa}\|_{3,2}^{\frac{5}{6}}.
\end{align*}
Using the bounds for the length and curvature \eqref{star}, Corollary~\ref{corinter}, Young inequality, and \eqref{stella1} we obtain
\begin{align*}
J_{3} & \leq C\| |f_{x}| \|_{L^{\infty}} (1 + \| \nabla_{s}^{3} \vec{\kappa}\|_{L^{2}}) + C \| \varphi \|_{L^{2}}^{2} + C\| \vec{\kappa}\|_{3,2}^{\frac{5}{3}} \\
& \leq \epsilon \int_{I} |\nabla_{s}^{3} \vec{\kappa}|^{2} ds +C_{\epsilon} (1+  \|\varphi\|_{L^{2}}^{2}) . 
\end{align*} 
Collecting all estimates found so far for $J_{1}$, $J_{2}$, $J_{3}$,  and choosing $\epsilon$ appropriately we find
\begin{align*}
\frac{d}{dt} &\left( \frac{1}{2}\int_{I} |\nabla_{s} \vec{\kappa}|^{2} ds\right) +
\frac{1}{2}\int_{I} |\nabla_{s}^{3} \vec{\kappa}|^{2} ds + \frac{1}{2}\int_{I} |\nabla_{s} \vec{\kappa}|^{2} ds \\
&\leq C(1 +\| \varphi\|_{L^{2}}^{2}) + C(1 +\| \varphi\|_{L^{2}}^{2})\int_{I} | \nabla_{s} \vec{\kappa}|^{2} ds 
\end{align*}
where  $ C=C(n ,\lambda, \mathcal{D}_{\lambda}(f_{0}))$. 
On the other hand using again Lemma~\ref{lemineqsh} we can write 
$$\int_{I} | \nabla_{s} \vec{\kappa}|^{2} ds  = \int_{I} |P_{2}^{2,1} (\vec{\kappa})| ds \leq \epsilon \int_{I} |\nabla_{s}^{3} \vec{\kappa}|^{2} ds + C(\epsilon, n, \lambda, \mathcal{D}_{\lambda}(f_{0}) ) ,
$$
so that, upon choosing $\epsilon$ small enough, we can finally write 
\begin{align*}
\frac{d}{dt} &\left( \frac{1}{2}\int_{I} |\nabla_{s} \vec{\kappa}|^{2} ds\right) +
\frac{1}{4}\int_{I} |\nabla_{s}^{3} \vec{\kappa}|^{2} ds + \frac{1}{2}\int_{I} |\nabla_{s} \vec{\kappa}|^{2} ds \\
&\leq C(1 +\| \varphi\|_{L^{2}}^{2}) + C\| \varphi\|_{L^{2}}^{2}\int_{I} | \nabla_{s} \vec{\kappa}|^{2} ds 
\end{align*}
where  $ C=C(n ,\lambda, \mathcal{D}_{\lambda}(f_{0}))$.
It follows for $\xi(t) := e^{t}\int_{I} |\nabla_{s} \vec{\kappa}|^{2}(t) ds$ that
\begin{align*}
\xi'(t) \leq C e^{t}(1 +\| \varphi\|_{L^{2}}^{2}) + C \| \varphi\|_{L^{2}}^{2} \xi(t).
\end{align*}
Using that $\int_{0}^{t} e^{t'}(1 +\| \varphi\|_{L^{2}}^{2}) dt' \leq e^{t} + e^{t} \int_{0}^{t}\| \varphi\|_{L^{2}}^{2} dt' \leq C e^{t}$ by \eqref{wichtig} we infer
\begin{align*}
\xi(t) \leq \xi(0) + Ce^{t} + C\int_{0}^{t}\| \varphi \|_{L^{2}}^{2}(t') \xi(t')\,  dt'
\end{align*}
and a Gronwall Lemma gives that $\xi(t) \leq C (\xi(0) + Ce^{t})$, that is
\begin{align}\label{boundnablak}
 \sup_{[0,T)}\| \nabla_{s} \vec{\kappa}\|_{L^{2}}^{2}(t) \leq C_{1} =C_{1}(n, \lambda, \mathcal{D}_{\lambda}(f_{0}), f_{0}).
\end{align}
Note that the above bound  together with \eqref{wichtig}, \eqref{star}, Lemma~\ref{hilfsatz} yields
\begin{align}\label{boundk1}
 \sup_{[0,T)} \| \vec{\kappa}(t) \|_{L^{\infty}}& \leq C_{1} =C_{1}(n, \lambda, \mathcal{D}_{\lambda}(f_{0}), f_{0}),\\
 \label{boundk12}
 \sup_{[0,T)} \int_{0}^{t} \| \nabla_{s}^{3} \vec{\kappa} \|_{L^{2}}^{2} dt'  &\leq C_{1,1} =C_{1,1}(n, \lambda, \mathcal{D}_{\lambda}(f_{0}), f_{0},T).
\end{align}

\bigskip
\noindent \textbf{First Step-Part B: Bound on $\|\varphi\|_{L^{2}}$.}
Although we know already that the $L^{2}$-norm of the tangential component behaves well in time (in the sense of \eqref{wichtig}), we need to refine this information.
To that end we consider
\begin{align*}
&\frac{d}{dt} \left(\frac{1}{2} \int_{I} \varphi^{2} ds \right)=
\int_{I} \varphi \varphi_{t} ds + \frac{1}{2} \int_{I} \varphi^{2} (\varphi_{s} - \langle \vec{\kappa}, \vec{V} \rangle) ds\\
&\quad = \int_{I} \varphi \lambda (|f_{x}| \varphi_{s})_{s} - \lambda \varphi (\langle \vec{\kappa}, \vec{V} \rangle |f_{x}|)_{s} +(\langle \vec{\kappa}, \vec{V} \rangle - \varphi_{s})  \varphi^{2} ds + \frac{1}{2} \int_{I} \varphi^{2} (\varphi_{s} - \langle \vec{\kappa}, \vec{V} \rangle) ds
\end{align*}
where we have used \eqref{a} and \eqref{PDEvarphi}.
Integration by parts and the fact that $\int_{I} \varphi^{2} \varphi_{s} ds =0 $ (this can be seen using integration by parts) yields
\begin{align*}
\frac{d}{dt} \left(\frac{1}{2} \int_{I} \varphi^{2} ds \right)  + \lambda \int_{I} (\partial_{s} \varphi)^{2} |f_{x}| ds &= \lambda \int_{I} \varphi_{s}\langle \vec{\kappa}, \vec{V} \rangle |f_{x}| ds
 +\frac{1}{2} \int_{I} \langle \vec{\kappa}, \vec{V} \rangle \varphi^{2} ds\\
 &=A_{1}+A_{2}.
\end{align*}
Using the $L^{\infty}$-bound on the curvature \eqref{boundk1} we can write
\begin{align*}
A_{1} &=\lambda \int_{I} \varphi_{s}\langle \vec{\kappa}, \vec{V} \rangle |f_{x}| ds \leq
\epsilon \lambda \int_{I} (\partial_{s} \varphi)^{2} |f_{x}| ds + \lambda C_{\epsilon} \int_{I}|\vec{V}|^{2} |f_{x}| ds\\
& \leq\epsilon \lambda \int_{I} (\partial_{s} \varphi)^{2} |f_{x}| ds +  \lambda C_{\epsilon} (1 + \| \varphi\|_{L^{2}}^{2}) \|\vec{V}\|_{L^{2}}^{2}
\end{align*}
where we have used \eqref{stella1} in the last step. For the second term $A_{2}$ we obtain,  using \eqref{boundk1}, \eqref{stella1}, and \eqref{V}, that
\begin{align*}
A_{2} &=\frac{1}{2} \int_{I} \langle \vec{\kappa}, \vec{V} \rangle \varphi^{2} ds
\leq C\int_{I} |P_{2}^{2,2} (\vec{\kappa})| \varphi^{2} ds +  C (1 +
\| |f_{x}|\|_{\infty}) \int_{I}\varphi^{2} ds \\
& \leq C  (\| |P_{2}^{2,2} (\vec{\kappa})| \|_{\infty} +1 + \| \varphi\|_{L^{2}}^{2} ) \int_{I}\varphi^{2} ds.
\end{align*}
Using \eqref{trickboundary2}, \eqref{interineq1sh2}, \eqref{star}, we obtain
\begin{align*}
|P_{2}^{2,2} (\vec{\kappa})| &\leq C \int_{I} |P_{2}^{3,3} (\vec{\kappa})|+ |P_{2}^{2,2} (\vec{\kappa})| ds 
 \leq C (
\| \nabla_{s}^{3} \vec{\kappa} \|_{L^{2}}^{2} +1 ) , 
\end{align*}
so that
\begin{align*}
A_{2} \leq C(1+ \| \nabla_{s}^{3} \vec{\kappa} \|_{L^{2}}^{2} + \| \varphi\|_{L^{2}}^{2} ) \int_{I}\varphi^{2} ds.
\end{align*}
Putting all estimates together we obtain (recall that $|f_{t}|^{2} =|\vec{V}|^{2} +\varphi^{2}$)
\begin{align*}
A_{1}+A_{2} \leq \epsilon \lambda \int_{I} (\partial_{s} \varphi)^{2} |f_{x}| ds + C_{\epsilon} \|\vec{V}\|_{L^{2}}^{2} + C_{\epsilon} (
\| \nabla_{s}^{3} \vec{\kappa} \|_{L^{2}}^{2} +1 + \|f_{t} \|_{L^{2}}^{2})  \int_{I}\varphi^{2} ds
\end{align*}
with $C_{\epsilon}= C_{\epsilon}(n, \lambda, \mathcal{D}_{\lambda}(f_{0}), f_{0})$.
Choosing $\epsilon$ appropriately we can write
\begin{align*}
\frac{d}{dt} &\left(\frac{1}{2} \int_{I} \varphi^{2} ds \right) + \frac{1}{2} \int_{I} \varphi^{2} ds   + \frac{\lambda}{2} \int_{I} (\partial_{s} \varphi)^{2} |f_{x}| ds \\ &\leq  C( \|\vec{V}\|_{L^{2}}^{2} + \| \varphi \|_{L^{2}}^{2}) +
C ( \| \nabla_{s}^{3} \vec{\kappa} \|_{L^{2}}^{2} + \|f_{t} \|_{L^{2}}^{2})  \int_{I}\varphi^{2} ds
\end{align*}
where $C =C(n, \lambda, \mathcal{D}_{\lambda}(f_{0}), f_{0})$.
Recalling \eqref{wichtig} and \eqref{boundk12}, a Gronwall argument (as performed in First Step -Part A)  gives
\begin{align}\label{boundphi}
\sup_{ t \in [0,T)} \| \varphi\|_{L^{2}}^{2}(t) \leq \tilde{C}_{1} = \tilde{C}_{1}( n, \lambda, \mathcal{D}_{\lambda}(f_{0}), f_{0}, T). 
\end{align}
Note that the dependence of the constant on $T$ is caused by \eqref{boundk12}.
As a consequence we obtain also
\begin{align}\label{aiuto2}
\sup_{ t \in [0,T)} \int_{0}^{t} \int_{I} (\partial_{s} \varphi)^{2} |f_{x}| ds dt' 
\leq \tilde{C}_{1} =\tilde{C}_{1}( n, \lambda, \mathcal{D}_{\lambda}(f_{0}), f_{0}, T).
\end{align}
Moreover recalling \eqref{stella1}, the definition of $\vec{w}$, and \eqref{boundk1} we can state
\begin{align}\label{boundw}
\sup_{ t \in [0,T)} \| \,|f_{x}| \|_{\infty}  +  \| \vec{w} \|_{\infty} \leq \tilde{C}_{1}=\tilde{C}_{1}( n, \lambda, \mathcal{D}_{\lambda}(f_{0}), f_{0}, T).
\end{align}
From the expression \eqref{nablaw} together with  the bounds \eqref{boundnablak}, \eqref{boundk1}, \eqref{boundphi} and \eqref{boundw} it follows
\begin{align}\label{boundnablaw}
\sup_{ t \in [0,T)} \| \nabla_{s}w\|_{L^{2}}^{2}(t) \leq \tilde{C}_{1}=\tilde{C}_{1}( n, \lambda, \mathcal{D}_{\lambda}(f_{0}), f_{0}, T).
\end{align}
Finally note that since $f_{xx}= |f_{x}|^{2} \vec{\kappa} + \frac{\varphi}{\lambda} |f_{x}| \tau $ we derive
\begin{align*}
\sup_{ t \in [0,T)}  \int_{I} |f_{xx}|^{2} dx \leq \tilde{C}_{1}=\tilde{C}_{1}( n, \lambda, \mathcal{D}_{\lambda}(f_{0}), f_{0}, T).
\end{align*}


\noindent \textbf{Intermezzo: bound from below for the length element $|f_{x}|$.}
In Section \ref{sec21} we  computed
\begin{align*}
\partial_{t} (|f_{x}|) 
& =\frac{\lambda}{|f_{x}|} (|f_{x}|)_{xx} + \lambda (|f_{x}|)_{x} \left( \frac{1}{|f_{x}|}\right)_{x}
- \langle \vec{\kappa}, \vec{V}\rangle |f_{x}|.
\end{align*}
Hence, using  \cite[Lemma 2.1.3]{mantegazza} and the uniform bound on the curvature \eqref{boundk1} we infer that $g(t):= \min_{I} |f_{x} (x,t)| $ is a positive map that satisfies 
\begin{align*}
g_{t} \geq- g \langle\vec{\kappa}, \vec{V} \rangle \geq  -C g  \| \vec{V}\|_{L^{\infty}}.
\end{align*}
 Since $\nabla_{s} \vec{V}= -\nabla_{s}^{3} \vec{\kappa} + P_{3}^{1,1}(\vec{\kappa}) +\lambda \nabla_{s}\vec{w}$, Lemma~\ref{hilfsatz}, \eqref{star},  \eqref{boundk1}, \eqref{boundnablak}, and \eqref{boundnablaw} yield 
 \begin{align}\label{aiutoV}
 \| \vec{V} \|_{\infty}  \leq C (\| \vec{V} \|_{L^{2}} + \| \nabla_{s} \vec{V} \|_{L^{2}})  \leq
 C (1+ \| \vec{V} \|_{L^{2}}^{2} + \| \nabla_{s}^{3} \vec{\kappa} \|_{L^{2}}^{2} ),
 \end{align}
 so that, upon recalling \eqref{boundk12} and \eqref{wichtig}, we can state that
 $$ \sup_{ t \in [0, T)} \int_{0}^{t} \| \vec{V} \|_{\infty}(t') dt' \leq C$$
 where $C=C ( n, \lambda, \mathcal{D}(f_{0}), f_{0}, T)$.
 Thus, integrating in time  the inequality
 $(\ln g)_{t} \geq -C  \| \vec{V}\|_{L^{\infty}}$
 gives 
 \begin{align*}
 g(t) \geq e^{-C} g(0)
 \end{align*}
 with $C=C ( n, \lambda, \mathcal{D}_{\lambda}(f_{0}), f_{0}, T)$. This yields 
 \begin{align}\label{3.11}
 \inf_{ t \in [0, T)} |f_{x}| \geq  C=C ( n, \lambda, \mathcal{D}_{\lambda}(f_{0}), f_{0}, T).
 \end{align}
An important consequence is that \eqref{aiuto2} now yields
\begin{align}\label{aiuto3}
\sup_{ t \in [0,T)} \int_{0}^{t} \int_{I} (\partial_{s} \varphi)^{2}  ds dt' 
\leq \tilde C_{1}= \tilde{C}_{1} ( n, \lambda, \mathcal{D}_{\lambda}(f_{0}), f_{0}, T).
\end{align}

\bigskip
\noindent \textbf{Induction step:}
Assume that for some  $ m \geq 1$ we have the following induction hypothesis: 
\begin{align} \label{IP} \tag{IP}
\left \{
\begin{array}{r} 
\sup_{t \in [0,T)} ( \sum_{j=0}^{m} \|\nabla_{s}^{j} \vec{\kappa} \|_{L^{2}}   
   + \int_{0}^{t} \| \nabla_{s}^{m+2} \vec{\kappa} \|_{L^{2}}^{2} dt' ) \leq C_{m}, \\
   \sup_{t \in [0,T)} ( \sum_{j=0}^{m-1} \| \partial_{s}^{j}\varphi \|_{L^{2}}  +\int_{0}^{t} \| \partial_{s}^{m}\varphi \|_{L^{2}}^{2} dt' + \sum_{j=0}^{m} \| \nabla_{s}^{j} \vec{w}\|_{L^{2}}) \leq \tilde{C}_{m}, 
   \end{array} \right.
   \end{align}
with $C_{m}=C_{m}( n, \lambda, \mathcal{D}(f_{0}), f_{0}, T) $, $\tilde{C}_{m}=\tilde{C}_{m}( n, \lambda, \mathcal{D}(f_{0}), f_{0}, T) $.

Note that by Lemma~\ref{hilfsatz} this means in particular that
\begin{align}\label{IP2} 
\left \{
\begin{array}{r} 
\sup_{t \in [0,T)} ( \sum_{j=0}^{m-1} \|\nabla_{s}^{j} \vec{\kappa} \|_{L^{\infty}}  + 
   \| \nabla_{s}^{m}\vec{\kappa}\|_{L^{2}}+ \int_{0}^{t} \| \nabla_{s}^{m+2} \vec{\kappa} \|_{L^{2}}^{2} dt' ) \leq C_{m}, \\
   \sup_{t \in [0,T)} ( \sum_{j=0}^{m-2} \| \partial_{s}^{j}\varphi \|_{L^{\infty}}  + 
   \| \partial_{s}^{m-1}\varphi\|_{L^{2}}+\int_{0}^{t} \| \partial_{s}^{m}\varphi \|_{L^{2}}^{2} dt') \leq \tilde{C}_{m}, \\
 \sup_{t\in[0,T)} (\| \nabla_{s}^{m} \vec{w}\|_{L^{2}} +  \sum_{j=0}^{m-1} \| \nabla_{s}^{j} \vec{w}\|_{L^{\infty}}) \leq \tilde{C}_{m}.
 \end{array} \right.
 \end{align}
 
\noindent \textbf{Induction Step - Part A:}
Using Lemma~\ref{lempartint} with $\vec{\phi}:=\nabla_{s}^{m+1}\vec{\kappa}$, Lemma~\ref{evolcurvature}, and exploiting the fact that $ \varphi \langle \nabla_{s}^{m+1} \vec{\kappa}, \nabla_{s}^{m+2} \vec{\kappa} \rangle +  \frac{1}{2} \varphi_{s} |\nabla_{s}^{m+1} \vec{\kappa}|^{2} = \partial_{s} \left(  \frac{1}{2} \varphi |\nabla_{s}^{m+1} \vec{\kappa}|^{2}\right)$
we obtain
\begin{align*}
\frac{d}{dt} &\left( \frac{1}{2}\int_{I} |\nabla_{s}^{m+1} \vec{\kappa}|^{2} ds\right) +
\int_{I} |\nabla_{s}^{m+3} \vec{\kappa}|^{2} ds + \frac{1}{2}\int_{I} |\nabla_{s}^{m+1} \vec{\kappa}|^{2} ds \\
&=\int_{I}\langle (\nabla_{t}+ \nabla_{s}^{4}) \nabla_{s}^{m+1} \vec{\kappa},\nabla_{s}^{m+1} \vec{\kappa} \rangle 
+ \frac{1}{2} \varphi_{s} |\nabla_{s}^{m+1} \vec{\kappa}|^{2} ds \\
& \quad - \frac{1}{2} \int_{I} |\nabla_{s}^{m+1} \vec{\kappa}|^{2} \langle  \vec{\kappa}, \vec{V} \rangle ds + \frac{1}{2}\int_{I} |\nabla_{s}^{m+1} \vec{\kappa}|^{2} ds\\
&= \int_{I} \langle  P^{m+3,m+3}_3 (\vec{\kappa})+ \lambda (\nabla_s^{m+3} \vec{w} + Q^{m+1,m+1}_3 (\vec{\kappa}, \vec{w})) +P^{m+1,m+1}_5 (\vec{\kappa}), \nabla_{s}^{m+1} \vec{\kappa} \rangle ds\\
& \quad - \frac{1}{2} \int_{I} |\nabla_{s}^{m+1} \vec{\kappa}|^{2} \langle  \vec{\kappa}, \vec{V} \rangle ds
+ \frac{1}{2}\int_{I} |\nabla_{s}^{m+1} \vec{\kappa}|^{2} ds
\\
& =\int_{I} P_{4}^{2m+4,m+3} (\vec{\kappa}) + P^{2m+2,m+1}_6 (\vec{\kappa})  + P_{2}^{2m+2,m+1} (\vec{\kappa})
+ P_{4}^{2m+4,m+1} (\vec{\kappa}) ds \\
& \quad+
 \lambda \int_{I} \langle \nabla_s^{m+3} \vec{w} + Q^{m+1,m+1}_3 (\vec{\kappa}, \vec{w}), \nabla_{s}^{m+1} \vec{\kappa} \rangle ds\\
 & = \int_{I}  P_{4}^{2m+4,m+3} (\vec{\kappa}) + P^{2m+2,m+1}_6 (\vec{\kappa})  + P_{2}^{2m+2,m+1} (\vec{\kappa})  ds \\
 & \quad +
 \lambda \int_{I} \langle \nabla_{s}^{m+1} \vec{w}, \nabla_{s}^{m+3} \vec{\kappa} \rangle ds + \lambda  \int_{I} \langle Q^{m+1,m+1}_3 (\vec{\kappa}, \vec{w}), \nabla_{s}^{m+1} \vec{\kappa} \rangle ds\\
& =J_{1}+J_{2}+J_{3},
\end{align*}
where we have used integration by parts in the last step and absorbed the terms $ P_{4}^{2m+4,m+1} (\vec{\kappa}) $ into $ P_{4}^{2m+4,m+3} (\vec{\kappa}) $.
Using \eqref{star} we find by applying Lemma~\ref{lemineqsh}
\begin{align*}
|J_{1}|&=|\int_{I}  P_{4}^{2m+4,m+3} (\vec{\kappa}) + P^{2m+2,m+1}_6 (\vec{\kappa})  + P_{2}^{2m+2,m+1} (\vec{\kappa})    ds |\\
& \leq \epsilon \int_{I} |\nabla_{s}^{m+3} \vec{\kappa}|^{2} ds + C(\epsilon, n, \mathcal{D}_{\lambda}(f_{0}) ) .
\end{align*}
Using \eqref{cs1} and the expression for $\varphi$ we can write for $m \geq 1$
\begin{align}\label{nablamw}
\nabla_{s}^{m+1} \vec{w}&= \nabla_{s}^{m+1}( |f_{x}| \vec{\kappa}) = \sum_{r=0}^{m+1}\binom{m+1}{r} \partial_{s}^{m+1-r} (|f_{x}|) \nabla_{s}^{r} \vec{\kappa}\notag \\
&=|f_{x}| \nabla_{s}^{m+1} \vec{\kappa} + 
\frac{d_{m}}{\lambda}\varphi \nabla_{s}^{m} \vec{\kappa} \notag\\
& \quad +\frac{1}{\lambda}\left(  
 d_{m-1} \partial_{s} \varphi  \nabla_{s}^{m-1} \vec{\kappa} 
 +    \ldots +d_{2} \partial_{s}^{m-2} \varphi\nabla_{s}^{2} \vec{\kappa} \right ) \notag \\
& \quad  + \frac{d_{1}}{\lambda}\partial_{s}^{m-1} \varphi \nabla_{s}\vec{\kappa}
 + \frac{1}{\lambda}\vec{\kappa}\partial_{s}^{m} \varphi 
\notag\\
& =|f_{x}| \nabla_{s}^{m+1} \vec{\kappa}+\frac{d_{m} }{\lambda}\varphi \nabla_{s}^{m} \vec{\kappa} + \frac{d_{1}}{\lambda}\partial_{s}^{m-1} \varphi \nabla_{s}\vec{\kappa}
 + \frac{1}{\lambda}\vec{\kappa}\partial_{s}^{m} \varphi  +\frac{1}{\lambda} W
\end{align}
(for appropriate coefficients $d_{j}$ which we do not  specify for notation purposes)
where 
$$ W:=   d_{m-1} \partial_{s} \varphi  \nabla_{s}^{m-1} \vec{\kappa}  + \ldots +d_{2} \partial_{s}^{m-2} \varphi\nabla_{s}^{2} \vec{\kappa} $$
and with the convention that $W=0$ if $m=1,2$. Note that if $m \geq 3$ then $|W| \leq C$ by \eqref{IP2}, so in what follows we can treat $W$ as a bounded term.
Therefore we can write
\begin{align*}
J_{2} &=\lambda \int_{I} \langle \nabla_{s}^{m+1} \vec{w}, \nabla_{s}^{m+3} \vec{\kappa} \rangle ds
=\lambda \int_{I} |f_{x}| \langle \nabla_{s}^{m+1} \vec{\kappa} ,  \nabla_{s}^{m+3} \vec{\kappa}\rangle ds +
 \int_{I}  \partial_{s}^{m}\varphi \langle   \vec{\kappa},  \nabla_{s}^{m+3} \vec{\kappa}\rangle ds\\
 &\quad 
 + \int_{I}  d_{1} \partial_{s}^{m-1}\varphi  \langle \nabla_{s} \vec{\kappa},  \nabla_{s}^{m+3} \vec{\kappa}\rangle ds 
 +\int_{I} d_{m} \varphi \langle \nabla_{s}^{m}  \vec{\kappa}, \nabla_{s}^{m+3} \vec{\kappa} \rangle ds+
 \int_{I} \langle W,\nabla_{s}^{m+3} \vec{\kappa}\rangle ds  \\
 &=J_{2,1}+J_{2,2} + J_{2,3} + J_{2,4} + J_{2,5}.
\end{align*}
Using \eqref{star}, \eqref{boundw}, \eqref{IP2} , Young inequality, and Lemma~\ref{lemineqsh},  one can  verify that
\begin{align*}
J_{2,1} + J_{2,5} \leq \epsilon \int_{I} |\nabla_{s}^{m+3} \vec{\kappa}|^{2} ds + C(\epsilon, n, \mathcal{D}_{\lambda}(f_{0}), f_{0},T ) .
\end{align*}
Young inequality, \eqref{boundk1}, \eqref{IP}, \eqref{IP2}, and Lemma~\ref{hilfsatz} yield
\begin{align*}
J_{2,2}& + J_{2,3} + J_{2,4}\\
 &\leq  \epsilon \int_{I} |\nabla_{s}^{m+3} \vec{\kappa}|^{2} ds + C_{\epsilon} 
(\| \partial_{s}^{m} \varphi \|_{L^{2}}^{2} + \| \partial_{s}^{m-1}  \varphi\|_{L^{\infty}}^{2}
\| \nabla_{s} \vec{\kappa} \|_{L^{2}}^{2} + \|   \varphi\|_{L^{\infty}}^{2}
\| \nabla_{s}^{m} \vec{\kappa} \|_{L^{2}}^{2} 
)\\
&\leq  \epsilon \int_{I} |\nabla_{s}^{m+3} \vec{\kappa}|^{2} ds + C_{\epsilon} 
(\| \partial_{s}^{m} \varphi \|_{L^{2}}^{2} + \| \partial_{s}^{m-1}  \varphi\|_{L^{2}}^{2}
 + \|   \varphi\|_{L^{2}}^{2} +  \|  \partial_{s} \varphi\|_{L^{2}}^{2}
)\\
&\leq \epsilon \int_{I} |\nabla_{s}^{m+3} \vec{\kappa}|^{2} ds + C_{\epsilon} 
(\| \partial_{s}^{m} \varphi \|_{L^{2}}^{2} +\| \partial_{s}^{m-1} \varphi \|_{L^{2}}^{2} + 1).
\end{align*}
Next, observe that (neglecting here for simplicity the coefficients multiplying each term)
\begin{align*}
Q^{m+1,m+1}_3 (\vec{\kappa}, \vec{w}) &= \nabla_{s}^{m+1} \vec{w} *\vec{\kappa}* \vec{\kappa} +
\nabla_{s}^{m} \vec{w} *\vec{\kappa}* \nabla_{s}\vec{\kappa} + R_{m+1} + \vec{w} * \nabla_{s}\vec{\kappa}* \nabla_{s}^{m}\vec{\kappa} \\
& \quad + \nabla_{s}\vec{w} * \nabla_{s}^{m}\vec{\kappa}* \vec{\kappa}   +  \vec{w} * \vec{\kappa}* \nabla_{s}^{m+1}\vec{\kappa} 
\end{align*}
where $ R_{m+1} $ contains all terms of type $\nabla_{s}^{i_{1}} \vec{w} * \nabla_{s}^{i_{2}}\vec{\kappa} * \nabla_{s}^{i_{2}}\vec{\kappa}$ with $i_{1}+i_{2}+i_{3} =m+1$ and $i_{1} \leq m-1$, $i_{2}, i_{3} \leq m-1$ (In case $m=1$, $R_{m+1}=0$). Due to \eqref{IP2} we see that $|R_{m+1}| \leq C$.
Thus, using \eqref{IP2} we can write
\begin{align*}
|\lambda\langle Q^{m+1,m+1}_3 (\vec{\kappa}, \vec{w}), \nabla_{s}^{m+1} \vec{\kappa} \rangle |
&\leq C(|\nabla_{s}^{m+1}\vec{w}||\nabla_{s}^{m+1} \vec{\kappa}|+ 
|\nabla_{s}^{m} \vec{w} *\vec{\kappa}* \nabla_{s}\vec{\kappa}* \nabla_{s}^{m+1} \vec{\kappa}|\\& +
|\nabla_{s}^{m+1} \vec{\kappa}|
+ |\vec{w} * \nabla_{s}\vec{\kappa}* \nabla_{s}^{m}\vec{\kappa} * \nabla_{s}^{m+1} \vec{\kappa}|\\
&
 + |\nabla_{s}\vec{w} * \nabla_{s}^{m}\vec{\kappa}* \vec{\kappa} *\nabla_{s}^{m+1} \vec{\kappa}| +|\nabla_{s}^{m+1} \vec{\kappa}|^{2} 
).
\end{align*}
Taking into account the expression \eqref{nablamw} derived above, \eqref{boundw}, and \eqref{IP2} we obtain
\begin{align*}
|\nabla_{s}^{m+1}\vec{w}||\nabla_{s}^{m+1} \vec{\kappa}| &\leq C (|\nabla_{s}^{m+1} \vec{\kappa}|^{2} + | \varphi| |\nabla_{s}^{m} \vec{\kappa}| |\nabla_{s}^{m+1} \vec{\kappa}| + |\nabla_{s}^{m+1} \vec{\kappa}| \\
&\qquad+ |\partial_{s}^{m-1} \varphi||\nabla_{s} \vec{\kappa} ||\nabla_{s}^{m+1} \vec{\kappa}| + |\partial_{s}^{m}\varphi| |\nabla_{s}^{m+1} \vec{\kappa}|)\\
& \leq C (|\nabla_{s}^{m+1} \vec{\kappa}|^{2} + \| \varphi \|_{\infty}^{2} |\nabla_{s}^{m} \vec{\kappa}|^{2} + 1 +  \| \partial_{s}^{m-1}\varphi \|_{\infty}^{2} |\nabla_{s} \vec{\kappa}|^{2} 
+ | \partial_{s}^{m} \varphi |^{2}).
\end{align*}
Similarly
\begin{align*}
|\nabla_{s}\vec{w} * \nabla_{s}^{m}\vec{\kappa}* \vec{\kappa} *\nabla_{s}^{m+1} \vec{\kappa}| &\leq C(\| \varphi \|_{\infty}^{2} |\nabla_{s}^{m} \vec{\kappa}|^{2} + |\nabla_{s}^{m+1} \vec{\kappa}|^{2} + |P_{3}^{2m+2, m+1}(\vec{\kappa})|
),\\
|\vec{w} * \nabla_{s}\vec{\kappa}* \nabla_{s}^{m}\vec{\kappa} * \nabla_{s}^{m+1} \vec{\kappa}| &\leq C|P_{3}^{2m+2, m+1}(\vec{\kappa})|,
\end{align*}
and for $m \geq 2$ (note that if $m=1$ the following term has already been dealt with, since for $m=1$ we have $\nabla_{s}^{m} \vec{w} *\vec{\kappa}* \nabla_{s}\vec{\kappa}* \nabla_{s}^{m+1} \vec{\kappa}= \nabla_{s}\vec{w} * \nabla_{s}^{m}\vec{\kappa}* \vec{\kappa} *\nabla_{s}^{m+1} \vec{\kappa}$ )
 \begin{align*}
 |\nabla_{s}^{m}& \vec{w} *\vec{\kappa}* \nabla_{s}\vec{\kappa}* \nabla_{s}^{m+1} \vec{\kappa}|\\
 & \leq
 C (|P_{3}^{2m+2, m+1}(\vec{\kappa})| + |\varphi| |\nabla_{s}^{m-1}\vec{\kappa}|| \nabla_{s}^{m+1}\vec{\kappa}| +|\nabla_{s}^{m+1}\vec{\kappa}| + |\partial_{s}^{m-1} \varphi|| \nabla_{s}^{m+1}\vec{\kappa}|)\\
 & \leq C (|P_{3}^{2m+2, m+1}(\vec{\kappa})| +
 \| \varphi \|_{\infty}^{2} |\nabla_{s}^{m-1} \vec{\kappa}|^{2} + 1 +  |\nabla_{s}^{m+1} \vec{\kappa}|^{2} 
+ | \partial_{s}^{m-1} \varphi |^{2} )
 \end{align*}
so that using  \eqref{IP}, \eqref{IP2}, \eqref{star},  Lemma~\ref{lemineqsh}  and Lemma~\ref{hilfsatz} we obtain that
\begin{align*}
J_{3} &=\int_{I}\lambda\langle Q^{m+1,m+1}_3 (\vec{\kappa}, \vec{w}), \nabla_{s}^{m+1} \vec{\kappa} \rangle  ds \leq \epsilon \int_{I} |\nabla_{s}^{m+3} \vec{\kappa}|^{2} ds + C_{\epsilon}\\
& \quad + C( \| \varphi\|_{L^{2}}^{2} + \| \partial_{s} \varphi\|_{L^{2}}^{2}+
\| \partial_{s}^{m-1}\varphi\|_{L^{2}}^{2} + \|\partial_{s}^{m}  \varphi\|_{L^{2}}^{2})\\
& \leq \epsilon \int_{I} |\nabla_{s}^{m+3} \vec{\kappa}|^{2} ds + C_{\epsilon} + C(\| \partial_{s}^{m-1}\varphi\|_{L^{2}}^{2} + \|\partial_{s}^{m}  \varphi\|_{L^{2}}^{2}).
\end{align*}
Collecting the estimates for $J_{1}$, $J_{2}$ and $J_{3}$ we find
 \begin{align*}
 \frac{d}{dt} &\left( \frac{1}{2}\int_{I} |\nabla_{s}^{m+1} \vec{\kappa}|^{2} ds\right) +
\int_{I} |\nabla_{s}^{m+3} \vec{\kappa}|^{2} ds + \frac{1}{2}\int_{I} |\nabla_{s}^{m+1} \vec{\kappa}|^{2} ds \\
&= J_{1}+ J_{2}+J_{3} \leq \epsilon \int_{I} |\nabla_{s}^{m+3} \vec{\kappa}|^{2} ds  + C_{\epsilon} ( 1 + \| \partial_{s}^{m-1}\varphi \|_{L^{2}}^{2}+ \| \partial_{s}^{m}\varphi \|_{L^{2}}^{2}).
 \end{align*}
 Choosing $\epsilon$ appropriately and  using   \eqref{IP}, a Gronwall Lemma yields
 \begin{align}\label{boundkm}
 \sup_{ t \in [0,T)} \| \nabla_{s}^{m+1} \vec{\kappa} \|_{L^{2}}  + \sup_{ t \in [0,T)} \int_{0}^{t} \| \nabla_{s}^{m+3} \vec{\kappa}\|_{L^{2}}^{2} dt'\leq C_{m+1} 
 \end{align}
 with $C_{m+1}=C_{m+1}(n,\lambda, \mathcal{D}_{\lambda}(f_{0}), T, f_{0}).
$
Together with Lemma~\ref{hilfsatz} and \eqref{IP} we also infer that
 \begin{align}\label{318}
 \sup_{ t \in [0,T)} \| \nabla_{s}^{m} \vec{\kappa} \|_{L^{\infty}}   \leq C_{m+1}=C_{m+1}(n,\lambda, \mathcal{D}_{\lambda}(f_{0}), T, f_{0}).
  \end{align}
 
\noindent \textbf{Induction Step - Part B:}
We compute
 \begin{align*}
&\frac{d}{dt} \left(\frac{1}{2} \int_{I} (\partial_{s}^{m}\varphi)^{2} ds \right)=
\int_{I} \partial_{s}^{m}\varphi\,  \partial_{t} \partial_{s}^{m}\varphi ds + \frac{1}{2} \int_{I} (\partial_{s}^{m}\varphi)^{2} (\varphi_{s} - \langle \vec{\kappa}, \vec{V} \rangle) ds\\
& = \int_{I} \partial_{s}^{m}\varphi \left(
(\lambda |f_{x}|  \partial_{s}^{m+1} \varphi )_{s} 
- (  \lambda |f_{x}| \partial_{s}^{m} (\langle \vec{\kappa}, \vec{V} \rangle))_{s} +M_{2}^{m,m} (\langle\vec{\kappa}, \vec{V}\rangle, \varphi) +B_{2}^{m+1,m+1} (\varphi)
 \right) ds \\
&\quad + \frac{1}{2} \int_{I} (\partial_{s}^{m}\varphi)^{2} (\varphi_{s} - \langle \vec{\kappa}, \vec{V} \rangle) ds
\end{align*}
where we have used \eqref{a} and \eqref{PDEvarphims}. Integration by parts gives
\begin{align*}
 &\frac{d}{dt} \left(\frac{1}{2} \int_{I} (\partial_{s}^{m}\varphi)^{2} ds \right) + \int_{I} \lambda |f_{x}| (\partial_{s}^{m+1}\varphi)^{2}  ds= \int_{I} (\partial_{s}^{m+1}\varphi)   \lambda |f_{x}| \partial_{s}^{m} (\langle \vec{\kappa}, \vec{V} \rangle)  ds \\
& \qquad  +\int_{I}  (\partial_{s}^{m}\varphi ) \, M_{2}^{m,m} (\langle\vec{\kappa}, \vec{V}\rangle, \varphi) ds   +\int_{I}  (\partial_{s}^{m}\varphi ) \,B_{2}^{m+1,m+1} (\varphi)  ds =A_{1}+A_{2}+A_{3}.
 \end{align*}
 Note that exploiting the bound from below for the length element \eqref{3.11} we can write
 \begin{align}\label{eqfin}
 &\frac{d}{dt} \left(\frac{1}{2} \int_{I} (\partial_{s}^{m}\varphi)^{2} ds \right) + C\int_{I} (\partial_{s}^{m+1}\varphi)^{2}  ds + \frac{1}{2}\int_{I} \lambda |f_{x}| (\partial_{s}^{m+1}\varphi)^{2}  ds  \notag\\
 &\leq \int_{I} (\partial_{s}^{m+1}\varphi)   \lambda |f_{x}| \partial_{s}^{m} (\langle \vec{\kappa}, \vec{V} \rangle)  ds \notag 
 +\int_{I}  (\partial_{s}^{m}\varphi ) \, M_{2}^{m,m} (\langle\vec{\kappa}, \vec{V}\rangle, \varphi) ds  \notag  \\
 & \quad +\int_{I}  (\partial_{s}^{m}\varphi ) \,B_{2}^{m+1,m+1} (\varphi)  ds =A_{1}+A_{2}+A_{3}.
 \end{align}
 Observing that by  \eqref{IP2} and \eqref{318}
 \begin{align*}
 |\partial_{s}^{m} (\langle \vec{\kappa}, \vec{V} \rangle)| &\leq C (|\vec{V}| + \ldots  +|\nabla_{s}^{m} \vec{V}|) \leq C (1 +|\nabla_{s}^{m+1} \vec{\kappa}| +|\nabla_{s}^{m+2} \vec{\kappa}| + |\nabla_{s}^{m}\vec{w}|)\\
|\partial_{s}^{m-1} (\langle \vec{\kappa}, \vec{V} \rangle)| &\leq C (1+ |\nabla_{s}^{m+1} \vec{\kappa}|)\\
|\partial_{s}^{m-2} (\langle \vec{\kappa}, \vec{V} \rangle)| &\leq C,
 \end{align*}
and recalling \eqref{boundw} we derive immediately that
 \begin{align*}
 A_{1} &\leq \epsilon \int_{I} \lambda |f_{x}| (\partial_{s}^{m+1}\varphi)^{2}  ds + C_{\epsilon}
 \int_{I} (1 +|\nabla_{s}^{m+1} \vec{\kappa}|^{2} +|\nabla_{s}^{m+2} \vec{\kappa}|^{2} + |\nabla_{s}^{m}\vec{w}|^{2}) ds \\
 &\leq \epsilon \int_{I} \lambda |f_{x}| (\partial_{s}^{m+1}\varphi)^{2}  ds + C_{\epsilon} (1 + 
 \| \nabla_{s}^{m+2} \vec{\kappa} \|_{L^{2}}^{2})
 \end{align*}
 where we have used the induction hypthesis, \eqref{star},  and \eqref{boundkm} in the last step.
 Using the definition of $M^{m,m}_{2}$, the calculation above for $|\partial_{s}^{m} (\langle \vec{\kappa}, \vec{V} \rangle)|$, \eqref{318}, and \eqref{IP2}
 we infer that for any $m \geq 1$ there holds
 \begin{align*}
 |M_{2}^{m,m} (\langle\vec{\kappa}, \vec{V}\rangle, \varphi) |& \leq C (1 +|\nabla_{s}^{m+1} \vec{\kappa}| +|\nabla_{s}^{m+2} \vec{\kappa}| + |\nabla_{s}^{m}\vec{w}|) |\varphi| \\
 & \quad +
 C (1 + |\nabla_{s}^{m+1} \vec{\kappa}|) |\partial_{s} \varphi| + C + C( |\partial_{s}^{m-1} \varphi| + |\partial_{s}^{m} \varphi|).
 \end{align*}
 Thus together with Lemma~\ref{hilfsatz}, \eqref{IP2}, and \eqref{boundkm} we obtain
 \begin{align*}
 A_{2} &\leq C \big(\| \varphi \|_{L^{2}} \| \partial_{s}^{m}\varphi\|_{L^{2}} +\| \varphi \|_{L^{2}}\| \partial_{s}^{m}\varphi\|_{L^{2}}  \left[ \| \nabla_{s}^{m+1}\vec{\kappa}\|_{L^{\infty}} +  \| \nabla_{s}^{m+2}\vec{\kappa}\|_{L^{\infty}} \right]\\
& \quad  +\| \varphi \|_{L^{\infty}}\| \partial_{s}^{m}\varphi\|_{L^{2}} 
  \| \nabla_{s}^{m}\vec{w}\|_{L^{2}}  \big)\\
& \quad  + C \big( \left[\| \partial_{s}\varphi \|_{L^{2}}  + \| \partial_{s}^{m-1}\varphi \|_{L^{2}} \right] \| \partial_{s}^{m}\varphi\|_{L^{2}} \big) + C\| \partial_{s}^{m}\varphi\|_{L^{2}}^{2} \\
& \quad + C\| \nabla_{s}^{m+1}\vec{\kappa}\|_{L^{\infty}} \| \partial_{s}^{m}\varphi\|_{L^{2}} \| \partial_{s}\varphi\|_{L^{2}}
+C\\
& \leq C + C( \| \varphi \|_{L^{2}}^{2} + \| \partial_{s}\varphi \|_{L^{2}}^{2} + \| \partial_{s}^{m-1}\varphi\|_{L^{2}}^{2} +\| \partial_{s}^{m}\varphi\|_{L^{2}}^{2} ) \\& \quad +
C \| \varphi \|_{L^{2}}\| \partial_{s}^{m}\varphi\|_{L^{2}}  [\| \nabla_{s}^{m+1}\vec{\kappa}\|_{L^{2}} +  \| \nabla_{s}^{m+2}\vec{\kappa}\|_{L^{2}} +  \| \nabla_{s}^{m+3}\vec{\kappa}\|_{L^{2}}]
\\& \quad + C\| \partial_{s}^{m}\varphi\|_{L^{2}} \| \partial_{s}\varphi\|_{L^{2}}
(\| \nabla_{s}^{m+1}\vec{\kappa}\|_{L^{2}} + \| \nabla_{s}^{m+2}\vec{\kappa}\|_{L^{2}}  )\\
& \leq  C + C(\| \partial_{s}^{m-1}\varphi\|_{L^{2}}^{2} +\| \partial_{s}^{m}\varphi\|_{L^{2}}^{2} )
+ C (\| \nabla_{s}^{m+2}\vec{\kappa}\|_{L^{2}}^{2} +  \| \nabla_{s}^{m+3}\vec{\kappa}\|_{L^{2}}^{2}
) \int_{I} |\partial_{s}^{m}\varphi|^{2} ds.
 \end{align*}
 Finally, recalling the definition of $B^{m+1,m+1}_{2}(\varphi)$ we see that every term appearing in $ A_{3}$ is of type
 \begin{align*}
 \int_{I} (\partial_{s}^{m} \varphi) (\partial_{s}^{i_{1}} \varphi) (\partial_{s}^{i_{2}} \varphi) ds 
 \end{align*}
 with $i_{1}+i_{2} = m+1$ and $0 \leq i_{j} \leq m+1$.
 Each such term can be estimated using Lemma~\ref{hilfsatz}, \eqref{star} and  \eqref{IP2} as follows:
 if $i_{1}=m+1$ (hence $i_{2}=0$) then
\begin{align*}
 |\int_{I} (\partial_{s}^{m} \varphi) (\partial_{s}^{m+1} \varphi)  \,\varphi ds | & 
 \leq  
\epsilon\|\partial_{s}^{m+1} \varphi \|_{L^{2}}^{2} + C_{\epsilon}\| \varphi \|_{L^{\infty}}^{2} 
\int_{I} |\partial_{s}^{m} \varphi|^{2} ds\\
 &\leq  
\epsilon\|\partial_{s}^{m+1} \varphi \|_{L^{2}}^{2} + C_{\epsilon}(\| \varphi \|_{L^{2}}^{2} +\|\partial_{s} \varphi \|_{L^{2}}^{2})
\int_{I} |\partial_{s}^{m} \varphi|^{2} ds\\
&\leq  
\epsilon\|\partial_{s}^{m+1} \varphi \|_{L^{2}}^{2} + C_{\epsilon}(1 +\|\partial_{s}^{m} \varphi \|_{L^{2}}^{2})
\int_{I} |\partial_{s}^{m} \varphi|^{2} ds
 \end{align*} 
 where we have used \eqref{IP2} in the last step (note that for $m \geq 2$ we actually have  $\|\partial_{s} \varphi \|_{L^{2}}^{2} \leq C$).
 If $i_{1}=m$, $i_{2}=1$  with $m \geq 2$ then we write  using Lemma~\eqref{hilfsatz} and \eqref{IP2}
 \begin{align*}
 \int_{I} (\partial_{s}^{m} \varphi) (\partial_{s}^{m} \varphi)  (\partial_{s}\varphi) ds &\leq \|\partial_{s}\varphi\|_{L^{\infty}}\int_{I} |\partial_{s}^{m} \varphi|^{2} ds
\leq C (\|\partial_{s}\varphi\|_{L^{2}} +\|\partial_{s}^{2}\varphi\|_{L^{2}})\int_{I} |\partial_{s}^{m} \varphi|^{2} ds
\\
& \leq C(1 + \|\partial_{s}^{m-1}\varphi\|_{L^{2}}^{2} +\|\partial_{s}^{m}\varphi\|_{L^{2}}^{2})\int_{I} |\partial_{s}^{m} \varphi|^{2} ds.
 \end{align*}
 On the other hand if  $i_{1}=m$, $i_{2}=1$  with $m =1$ then integration by parts yields
 \begin{align*}
 \int_{I} (\partial_{s}^{m} \varphi) (\partial_{s}^{m} \varphi)  (\partial_{s}\varphi) ds &=
 \int_{I} (\partial_{s}^{m}\varphi)^{3} ds =- 2\int_{I} (\partial_{s}^{m} \varphi) (\partial_{s}^{m-1} \varphi) (\partial_{s}^{m+1} \varphi) ds 
 \\ 
 & \leq \epsilon\|\partial_{s}^{m+1} \varphi \|_{L^{2}}^{2} + C_{\epsilon} \| \partial_{s}^{m-1} \varphi\|_{L^{\infty}}^{2} \int_{I} |\partial_{s}^{m} \varphi|^{2} ds\\
 &\leq \epsilon\|\partial_{s}^{m+1} \varphi \|_{L^{2}}^{2} + C_{\epsilon}(\| \partial_{s}^{m-1} \varphi\|_{L^{2}}^{2} +\|\partial_{s}^{m} \varphi \|_{L^{2}}^{2})
\int_{I} |\partial_{s}^{m} \varphi|^{2} ds.
 \end{align*}
 If $i_{1}=m-1$, $i_{2}=2$ and $m=1$ then similar arguments as above yield
 \begin{align*}
 \int_{I} (\partial_{s}^{m} \varphi) (\partial_{s}^{m-1} \varphi)  (\partial_{s}^{2}\varphi) ds&=
 \int_{I} (\partial_{s}^{m} \varphi) (\partial_{s}^{m-1} \varphi)  (\partial_{s}^{m+1}\varphi) ds\\
  &\leq \epsilon\|\partial_{s}^{m+1} \varphi \|_{L^{2}}^{2} + C_{\epsilon}(\| \partial_{s}^{m-1} \varphi\|_{L^{2}}^{2} +\|\partial_{s}^{m} \varphi \|_{L^{2}}^{2})
\int_{I} |\partial_{s}^{m} \varphi|^{2} ds.
 \end{align*}
 On the other hand if $i_{1}=m-1$, $i_{2}=2$ and $m \geq2$ then
 \begin{align*}
 \int_{I} (\partial_{s}^{m} \varphi) (\partial_{s}^{m-1} \varphi)  (\partial_{s}^{2}\varphi) ds&\leq
 \|\partial_{s}^{2}\varphi \|_{L^{2}}^{2} + \|\partial_{s}^{m-1}\varphi\|_{L^{\infty}}^{2}\int_{I} |\partial_{s}^{m} \varphi|^{2} ds\\
 & \leq C + C(1+\|\partial_{s}^{m-1}\varphi\|_{L^{2}}^{2} +\|\partial_{s}^{m}\varphi\|_{L^{2}}^{2})\int_{I} |\partial_{s}^{m} \varphi|^{2} ds.
 \end{align*}
 Finally if $i_{1} \leq m-2$  $i_{2}=m+1-i_{1}$, then by \eqref{IP2} we know that $| \partial_{s}^{i_{1}} \varphi| \leq C$ and therefore we can write
 \begin{align*}
 |\int_{I} (\partial_{s}^{m} \varphi) (\partial_{s}^{i_{1}} \varphi) (\partial_{s}^{i_{2}} \varphi) ds | &\leq C \|\partial_{s}^{i_{2}} \varphi \|_{L^{2}} 
\|\partial_{s}^{m} \varphi \|_{L^{2}}  
 \leq  \|\partial_{s}^{i_{2}} \varphi \|_{L^{2}}^{2} +
\int_{I} |\partial_{s}^{m} \varphi|^{2} ds\\
& \leq C  +
\int_{I} |\partial_{s}^{m} \varphi|^{2} ds
 \end{align*}
 where we have taken \eqref{IP} and $i_{2} < m$ into account.
 According to all considerations outlined so far we can state that
 \begin{align*}
 |A_{3}| &\leq 
\epsilon\|\partial_{s}^{m+1} \varphi \|_{L^{2}}^{2} + C_{\epsilon}(1+\| \partial_{s}^{m-1} \varphi\|_{L^{2}}^{2} +\|\partial_{s}^{m} \varphi \|_{L^{2}}^{2})
\int_{I} |\partial_{s}^{m} \varphi|^{2} ds 
 +
 C.
  \end{align*}
Fron \eqref{eqfin} together with the obtained estimates for $A_{1}$, $A_{2}$, $A_{3}$, \eqref{IP},
 and choosing $\epsilon$ appropriately we obtain
\begin{align*}
 &\frac{d}{dt} \left(\frac{1}{2} \int_{I} (\partial_{s}^{m}\varphi)^{2} ds \right) + C\int_{I} (\partial_{s}^{m+1}\varphi)^{2}  ds \\
& \leq  C(1 + 
 \| \nabla_{s}^{m+2} \vec{\kappa} \|_{L^{2}}^{2}) + C ( 1 +\|\partial_{s}^{m} \varphi \|_{L^{2}}^{2} +\| \nabla_{s}^{m+2}\vec{\kappa}\|_{L^{2}}^{2} +  \| \nabla_{s}^{m+3}\vec{\kappa}\|_{L^{2}}^{2}
) \int_{I} |\partial_{s}^{m}\varphi|^{2} ds.
  \end{align*}
A Gronwall argument that takes into account \eqref{IP} and \eqref{boundkm} finally yields
\begin{align*}
\sup_{t \in [0,T)} (  \| \partial_{s}^{m}\varphi \|_{L^{2}}  +\int_{0}^{t} \| \partial_{s}^{m+1}\varphi \|_{L^{2}}^{2} dt' ) \leq \tilde{C}_{m+1}, 
   \end{align*}
with  $\tilde{C}_{m+1}=\tilde{C}_{m+1}( n, \lambda, \mathcal{D}_{\lambda}(f_{0}), f_{0}, T) $.
The uniform $L^{2}$-bound for $\nabla_{s}^{m+1}\vec{w}$ follows from \eqref{nablamw} and the uniform bounds obtained so far. The induction step is now completed.

\bigskip

\noindent \textbf{Final steps:} As observed in \cite[Thm 3.1]{DKS}, for a function $h:I \to \R$ we have that  $\partial_{x}^{m} h =  |f_{x}|^{m} \partial_{s}^{m} h + P_{m} (|f_{x}|, \ldots, \partial_{x}^{m-1}(|f_{x}|), h, \ldots, \partial_{s}^{m-1}h)$ where $P_{m}$ is a polynomial. 
With $h=|f_{x}|$ and \eqref{tc} it follows
$\partial_{x}^{m} (|f_{x}|) =  \frac{1}{\lambda}|f_{x}|^{m} \partial_{s}^{m-1} \varphi  + P_{m} (|f_{x}|, \ldots, \partial_{x}^{m-1}(|f_{x}|), \varphi, \ldots, \partial_{s}^{m-2}\varphi)$, so that 
 taking into account \eqref{boundw},   and the uniform bounds obtained for $\varphi$ and its  derivatives (recall that \eqref{IP2} holds for any $m$),  we obtain uniform bounds for the derivatives of the length element in the original parametrization.
Similarly, using Lemma~\ref{lememb}  and \eqref{IP2} we can also state that on $[0, T)$  we have for any $m \in \mathbb{N}$
$$\| \partial_{x}^{m} f \|_{L^{\infty}} \leq  C(m, n, \lambda, \mathcal{D}(f_{0}), f_{0}, T). $$
Moreover $\| f \|_{L^{\infty}} \leq C(n, \lambda, \mathcal{D}(f_{0}), f_{0}, T)$. Therefore we can extend $f$ smoothly over $[0,T] \times I$ and  even beyond by short-time existence contradicting the maximality of $T$. This proves that the flow exists globally in time.




\bibliography{refNum}
\bibliographystyle{acm}

\end{document}